\documentclass[12pt]{amsart}
\usepackage{amsmath, amssymb, amsthm}
\usepackage{tikz-cd}
\usepackage{adjustbox}
\usepackage{thmtools}
\usepackage{mathabx}
\usepackage{geometry}
\usepackage{mathrsfs}
\usepackage{thm-restate}
\begin{document}
\newtheorem{theorem}{Theorem}[section]
\newtheorem{proposition}[theorem]{Proposition}
\newtheorem{lemma}[theorem]{Lemma}
\newtheorem{corollary}[theorem]{Corollary}
\newtheorem{definition}[theorem]{Definition}
\newtheorem*{theorem*}{Theorem}
\newtheorem{conjecture}[theorem]{Conjecture}

\newcommand{\la}{\longrightarrow}
\newcommand{\bt}{\bullet}
\newcommand{\HF}{\mathit{HF}}
\newcommand{\HFI}{\mathit{HFI}}
\newcommand{\CF}{\mathit{CF}}
\newcommand{\CFI}{\mathit{CFI}}
\newcommand{\HM}{\mathit{HM}}
\newcommand{\HMI}{\mathit{HMI}}
\newcommand{\CM}{\mathit{CM}}
\newcommand{\CMI}{\mathit{CMI}}
\newcommand{\CZ}{\mathit{CZ}}
\newcommand{\HZ}{\mathit{HZ}}
\newcommand\xleftrightarrow[1]{\mathbin{\ooalign{$\,\xrightarrow{#1}$\cr$\xleftarrow{\hphantom{#1}}\,$}}
}

\begin{abstract} 

Inspired by the Ozsváth-Szabó mixed invariant in ordinary Heegaard Floer theory, we define a mixed invariant $\Phi_{X, \mathfrak{s}}^{I}$ for closed, spin four-manifolds $(X, \mathfrak{s})$ using the cobordism maps on involutive Heegaard Floer homology. The invariant is well-defined whenever $b_{2}^{+}(X) > 4$. We furthermore construct an involutive Seiberg-Witten invariant that is well-defined whenever $b_{2}^{+}(X) > 3$. We show that these involutive invariants obstruct the existence of disjoint pairs of embedded surfaces which both violate the adjunction inequality. As an application, we find that $K3\#(S^2 \times S^2)$ contains no such pair.

\end{abstract}
\title[Involutive Floer Invariants]{Involutive Floer Invariants for Closed Four-Manifolds}
\author{Owen Brass}
\date{}
\maketitle
\section{Introduction}

In \cite{MR3649355}, Hendricks and Manolescu introduced involutive Heegaard Floer homology, a variant of Heegaard Floer homology which accounts for the natural conjugation symmetry on Heegaard diagrams given by swapping $\alpha$ and $\beta$ curves and reversing the orientation of the splitting surface.  They showed that, as in the original theory, their invariant fits into the framework of a (suitably modified) TQFT—namely, a spin cobordism between spin three-manifolds induces a map between their involutive Heegaard Floer homologies.  It was shown in \cite{MR5032324} that these cobordism maps are well-defined. (Throughout this paper, all Floer theories have $\mathbb{F} := \mathbb{Z}_{2}$ coefficients unless indicated otherwise, and all Floer invariants—including Seiberg-Witten invariants—are taken mod 2.)

Many of the significant four-dimensional applications of Ozsváth and Szabó's original theory came through their \textit{mixed invariant} $\Phi_{X, \mathfrak{s}}$ (defined in \cite{MR2113019}) which, like the Seiberg-Witten invariant $SW_{X, \mathfrak{s}}$, is associated to a four-manifold $X$ with $b_2^{+} > 1$ and equipped with a $\text{spin}^{c}$ structure $\mathfrak{s}$ (in fact, $\Phi_{X, \mathfrak{s}}$ and $SW_{X, \mathfrak{s}}$ are conjecturally equal). Such applications include re-proofs of the indecomposability of symplectic four-manifolds and the symplectic Thom conjecture (see \cite{MR2031164}), as well as detection of new exotic four-manifolds (see e.g., \cite{MR4577958}, \cite{levine2023newconstructionsinvariantsclosed}).

In this paper, we construct an analogue of the mixed invariant in the setting of involutive Heegaard Floer theory—an \textit{involutive mixed invariant} $\Phi_{X, \mathfrak{s}}^{I}$—defined for spin four-manifolds $(X, \mathfrak{s})$ with $b_{2}^{+} > 3$.  We show that it is independent of the choices made in its construction when $b_{2}^{+}(X) > 4$ and $X$ is closed (the closed condition eliminates dependence on a choice of framed path between basepoints; see Theorem 1.3 of \cite{MR5032324}).  When $b_{2}^{+} = 4$, $\Phi_{X, \mathfrak{s}}^{I}$ a priori depends on a choice of an \textit{involutively admissible cut} $N \subset X$ (see Definition~\ref{invcut}), and we write $\Phi_{X, \mathfrak{s}, N}^{I}$. We conjecture that it does not actually depend on $N$.

The involutive mixed invariant is closely related to the Ozsváth-Szabó mixed invariant.  In the following theorem, we treat elements $a_{0}U^m + a_{1}QU^n \in \mathbb{F}[U, Q]/Q^2 \cong \HFI^{-}(S^3)$ as vectors of the form $(a_{0}U^m, a_{1}U^n)$. 
\begin{theorem}\label{matrix} Let $(X, \mathfrak{s})$ be a closed spin four-manifold with $b_{2}^{+} > 4$.  Then
\[\Phi_{X, \mathfrak{s}}^{I} = \begin{pmatrix} \Phi_{X, \mathfrak{s}} & 0 \\ \Psi_{X, \mathfrak{s}} & \Phi_{X, \mathfrak{s}} \end{pmatrix} : \HFI^{-}(S^3) \to \HFI^{+}(S^3),\]
where $\Phi_{X, \mathfrak{s}}$ is the Ozsváth-Szabó mixed invariant, and $\Psi_{X, \mathfrak{s}}$ is a $U$-equivariant homomorphism.  By grading considerations, $\Psi_{X, \mathfrak{s}}$ is nonzero only if $\Phi_{X, \mathfrak{s}}$ is zero—in fact, only if the virtual dimension of the Seiberg-Witten moduli space $d(\mathfrak{s})$ is odd.

If $b_{2}^{+} = 4$, then the same formula holds for $\Phi_{X, \mathfrak{s}, N}^{I}$, where $\Psi_{X, \mathfrak{s}}$ now a priori depends on $N$.
\end{theorem}
Using the calculation of the involutive cobordism map induced by $(S^2 \times S^2, \mathfrak{s}_{0})$ due to Hendricks, Hom, Stoffregen, and Zemke (see \cite{MR5032324}, Proposition 14.1), where $\mathfrak{s}_{0}$ is the unique spin structure on $S^2 \times S^2$, we obtain the following stabilization formula:
\begin{theorem}\label{stabintro} Let $(X, \mathfrak{s})$ be a closed spin four-manifold with $b_{2}^{+}(X) > 3$.  Then:
\[\Phi_{X \# (S^2 \times S^2), \mathfrak{s}\#\mathfrak{s}_{0}}^{I} = Q\Phi_{X, \mathfrak{s}}.\]
If $b_{2}^{+}(X) = 3$ and $N$ is any involutively admissible cut of $X\#(S^2 \times S^2)$ which is contained entirely in $X$, then the same formula holds for $\Phi_{X\#(S^2 \times S^2), \mathfrak{s}\#\mathfrak{s}_{0}, N}^{I}$.
\end{theorem}

Many useful applications of the Seiberg-Witten invariant and the Ozsváth-Szabó mixed invariant come through adjunction results, such as the following result for embedded spheres due to Fintushel and Stern:

\begin{theorem} [\textup{\cite{MR1349567}, \cite[Remark 4.4]{MR4800943}}] Let $X$ be a closed four-manifold with $b_{2}^{+} > 1$.  Suppose that there exists an embedded sphere $S \subset X$ with $[S]^{2} \geq 0$ and $[S] \neq 0$.  Then $SW_{X, \mathfrak{s}} = \Phi_{X, \mathfrak{s}} = 0$ for all $\text{spin}^{c}$ structures $\mathfrak{s}$ on $X$.
\end{theorem}

The analogous statement for the involutive mixed invariant is demonstrably false, as Theorem~\ref{stabintro} implies the existence of four-manifolds with an $S^2 \times S^2$ summand for which $\Phi_{X, \mathfrak{s}}^{I}$ is nonvanishing, and $S^2 \times S^2$ contains embedded spheres of self-intersection $2n$ for any $n \geq 0$.  However, one may still ask whether Floer theoretic techniques can recover any adjunction information after stabilization. We prove the following weaker adjunction result:

\begin{theorem}\label{adjintro} Let $X$ be a smooth four-manifold with $b_{2}^{+}(X) > 4$.  Suppose that there exists a disjoint pair of smoothly embedded spheres $S_1, S_2 \subset X$ such that $[S_1]^2 > 0$ and $[S_2]^{2}> 0$.  Then $\Phi_{X, \mathfrak{s}}^{I} = 0$ for all spin structures $\mathfrak{s}$ on $X$.
\end{theorem}

As a corollary of Theorems~\ref{stabintro} and~\ref{adjintro}, we obtain:

\begin{theorem}\label{stabadjintro} Let $X$ be a smooth, closed four-manifold with $b_{2}^{+}(X) > 3$.  Suppose that there exists a disjoint pair of smoothly embedded spheres $S_1, S_2 \subset X\#(S^2 \times S^2)$ such that $[S_{1}]^2 > 0$ and $[S_{2}]^2 > 0$.  Then $\Phi_{X, \mathfrak{s}} = 0$ for all spin structures $\mathfrak{s}$ on $X$, where $\Phi_{X, \mathfrak{s}}$ is the Ozsváth-Szabó mixed invariant of the pair $(X, \mathfrak{s})$.
    
\end{theorem}
Unfortunately, in light of Baraglia's result that $SW_{X, \mathfrak{s}} = 0$ mod 2 whenever $b_{2}^{+}(X) > 3$ and $\mathfrak{s}$ is spin \cite{baraglia2023mod}, there is good reason to believe that Theorem~\ref{stabadjintro} would only be interesting if we could extend the hypothesis to the case $b_{2}^{+}(X) = 3$.  As it turns out, Baraglia's Pin(2)-equivariant Seiberg-Witten invariant is the correct tool for this job.

We show that a suitable truncation of Baraglia's invariant can be obtained through a similar construction to the involutive mixed invariant, using cobordism maps on involutive monopole Floer homology (due to Lin \cite{MR3968878}). We call this truncation the \textit{involutive Seiberg-Witten invariant} $SW_{X, \mathfrak{s}}^{I}$. Working on the monopole side has the virtue that invariance can be extended to lower values of $b_{2}^{+}$.

The invariant $SW_{X, \mathfrak{s}}^{I}$ satisfies the same formal properties as $\Phi_{X, \mathfrak{s}}^{I}$:

\begin{theorem}\label{matrixSW} Let $(X, \mathfrak{s})$ be a closed spin four-manifold with $b_{2}^{+} \geq 4$.  Then
\[SW_{X, \mathfrak{s}}^{I} = \begin{pmatrix} SW_{X, \mathfrak{s}} & 0 \\ h_{X, \mathfrak{s}} & SW_{X, \mathfrak{s}} \end{pmatrix} : \widehat{\HMI}(S^3) \to \widecheck{\HMI}(S^3),\]
where $h_{X, \mathfrak{s}}$ is a $U$-equivariant homomorphism.  By grading considerations, $h_{X, \mathfrak{s}}$ is nonzero only if $SW_{X, \mathfrak{s}}$ is zero—in fact, only if the dimension of the Seiberg-Witten moduli space $d(\mathfrak{s})$ is odd.
\end{theorem}

\begin{theorem}\label{stabintroSW} Let $(X, \mathfrak{s})$ be a closed spin four-manifold with $b_{2}^{+}(X) \geq 3$.  Then:
\[SW_{X \# (S^2 \times S^2), \mathfrak{s}\#\mathfrak{s}_{0}}^{I} = Q\cdot SW_{X, \mathfrak{s}}.\]
\end{theorem}

As a result, we conclude that Theorem~\ref{stabadjintro} holds for $b_{2}^{+}(X) = 3$ when $\Phi_{X, \mathfrak{s}}$ is replaced by $SW_{X, \mathfrak{s}}$. In fact, analysis of the bar flavor of the involutive monopole Floer chain complex allows for a generalization of the result to higher genus surfaces:
\begin{theorem}\label{SWstabadjintro} Let $X$ be a smooth, closed four-manifold with $b_{2}^{+}(X) > 2$.  Suppose that there exists a disjoint pair of smoothly embedded surfaces $S_1, S_2 \subset X\#(S^2 \times S^2)$ such that
\[[S_{1}]^2 > \max\{2g(S_{1}) - 2, 0\}\]
and
\[[S_{2}]^2 > \max\{2g(S_{2}) - 2, 0\}.\]
Then $SW_{X, \mathfrak{s}} = 0$ mod 2 for all spin structures $\mathfrak{s}$ on $X$. 
\end{theorem}

Henceforth, we will use the terminology that a smoothly embedded surface $S$ of genus $g$ in a spin four-manifold $(X, \mathfrak{s})$ \textit{violates the adjunction inequality} if
\[[S]^2 > \max\{2g(S) - 2, 0\}.\]
As a corollary of Theorem~\ref{SWstabadjintro}, we obtain:
\begin{corollary} If $X$ is a homotopy $K3$ surface, then there exists no disjoint pair of smoothly embedded surfaces which both violate the adjunction inequality in $X\#(S^2 \times S^2)$.
\end{corollary}
\begin{proof} By work of Morgan-Szabó \cite{MR1432806}, $SW_{X, \mathfrak{s}} = 1$ mod 2, where $\mathfrak{s}$ is the unique self-conjugate $\text{spin}^{c}$ structure on $X$.  The result then follows from Theorem~\ref{SWstabadjintro}.\end{proof}

We remark that other adjunction results for configurations of embedded surfaces in four-manifolds with vanishing Seiberg-Witten invariants have been obtained by Strle in \cite{MR2064429} and Konno in \cite{MR4407492}.

\subsection{Outline of the Paper}
In sections 2 and 3, we provide an overview of the Seiberg-Witten and Heegaard Floer theory background that will be used throughout the paper. In section 4, we define the involutive mixed map for cobordisms, and we show that it is an invariant when $b_{2}^{+} > 4$.  In section 5, we prove some basic properties of the invariant, including its relationship to the Ozsváth-Szabó invariant (Theorem~\ref{matrix}) and its ability to obstruct the existence of certain embedded pairs of spheres (Theorem~\ref{adjintro}).  In section 6, we prove the stabilization formula (Theorem~\ref{stabintro}) and deduce weak sphere adjunction for stabilized four-manifolds with high $b_{2}^{+}$ (Theorem~\ref{stabadjintro}). Sections 7 and 8 turn at last to Seiberg-Witten theory. In section 7, we define the involutive monopole invariant and show that it is equivalent to several other naturally arising invariants (from which we deduce that it is, in fact, an invariant). In section 8, we prove our main Seiberg-Witten adjunction result (Theorem~\ref{SWstabadjintro}).
\subsection{Acknowledgements}
I am immensely grateful to my undergraduate advisor, Ciprian Manolescu, for suggesting this idea and providing invaluable guidance, as well as for his patience during the writing process; to Irving Dai for teaching me most of what I know about Heegaard Floer theory; to Gary Guth for many helpful discussions concerning the involutive mixed invariant (which he knew about long before I did); to Lernik Asserian for organizing the summer research program during which this paper began to take form; and to my current advisor Peter Kronheimer for his helpful suggestions and for his encouragement. I would also like to thank Kristen Hendricks and Robert Lipshitz for helpful comments. The author was partially supported by a National Science Foundation Graduate Research Fellowship.

\section{Background on Seiberg-Witten Theory}

\subsection{Seiberg-Witten Invariants}

The Seiberg-Witten invariants $SW_{X, \mathfrak{s}}$ are $\mathbb{Z}$-valued invariants associated to a pair $(X, \mathfrak{s})$, where $X$ is a closed, oriented, Riemannian four-manifold with $b_{2}^{+}(X) > 1$ and $\mathfrak{s}$ is a $\text{spin}^{c}$ structure on $X$. Throughout this paper, we will only use the $\mathbb{Z}_2$-valued Seiberg-Witten invariants—so we may ignore issues of orientation—with the sole exception of our proof of Proposition~\ref{tech}, where we rely on certain known results about the monopole Floer chain complex with $\mathbb{Z}$-coefficients.
\subsubsection{The Seiberg-Witten Equations}

The group $\text{Spin}^{c}(n)$ is defined to be $(\text{Spin}(n) \times S^1)/\{\pm (1, 1)\}$. By a $\text{spin}^{c}$ (resp. spin) structure $\mathfrak{s}$ on $X$, we mean a choice of lift of the structure group of the tangent bundle $TX$ from $SO(4)$ to $\text{Spin}^{c}(4)$ (resp. $\text{Spin}(4)$). Note that a spin structure induces a $\text{spin}^{c}$ structure via the natural map $\text{Spin}(4) \to \text{Spin}^{c}(4)$.

There is a natural conjugation action $\mathfrak{s} \mapsto \overline{\mathfrak{s}}$ given by conjugating in the $S^1$ component of $\text{Spin}^{c}(4)$. We will repeatedly use the following fact: \textit{A $\text{spin}^{c}$ structure $\mathfrak{s}$ is isomorphic to its conjugate $\overline{\mathfrak{s}}$ if and only if it is induced by a spin structure $\sigma$.} We will refer to such $\text{spin}^{c}$ structures as \textit{self-conjugate}, or occasionally (in an abuse of terminology) simply as spin structures.

Given a $\text{spin}^{c}$ structure $\mathfrak{s}$ on a Riemannian four-manifold $X$, we denote by $S = S^{+} \oplus S^{-}$ the bundle of spinors associated to $\mathfrak{s}$ and by $\rho: T^{*}X \to \text{End}(S)$ the Clifford multiplication. The latter may be extended to forms using a Leibniz rule. The extension identifies the bundle of self-dual 2-forms $\Omega^{+}(X) \subset \Omega^2(X)$ isometrically with the bundle of traceless skew-adjoint endomorphisms of the positive spinor bundle $\mathfrak{su}(S^{+})$, and similarly identifies the bundle of anti-self-dual 2-forms $\Omega^{-}(X)$ with $\mathfrak{su}(S^{-})$.

The (perturbed) Seiberg-Witten equations are defined on the space $\mathcal{C}(X, \mathfrak{s}) = \mathcal{A} \times \Gamma(S^{+})$ of pairs $(A, \phi)$, where $A$ is a $\text{spin}^{c}$-connection on $S$ (that is, a connection for which parallel transport preserves the splitting $S = S^{+} \oplus S^{-}$) and $\phi$ is a section of the positive spinor bundle $S^{+}$. The space of $\text{spin}^{c}$-connections on $S$ is an affine space modeled on $\Omega^1(X; i\mathbb{R})$. 

A $\text{spin}^{c}$-connection $A$ naturally induces a connection $A^t$ on the complex line bundle $\Lambda^2S^{+}$. Moreover, $A$ gives rise to a Dirac operator $D_A: \Gamma(S) \to \Gamma(S)$ via the composition
\[\Gamma(S) \overset{A}{\la} \Gamma(T^*X \otimes S) \overset{\rho}{\la} \Gamma(S).\]
This operator exchanges $\Gamma(S^{+})$ and $\Gamma(S^{-})$, and we denote by $D_{A}^{+}: \Gamma(S^{+}) \to \Gamma(S^{-})$ the restriction to $\Gamma(S^{+})$.

Following \cite{Kronheimer_Mrowka_2007}, we write the Seiberg-Witten equations perturbed by an imaginary 2-form $\omega \in \Omega^2(X; i\mathbb{R})$ as follows:
\[\mathfrak{F}_{\omega}(A, \phi) = 0\]
where
\[\mathfrak{F}_{\omega}: \mathcal{C}(X, \mathfrak{s}) \to i\mathfrak{su}(S^{+}) \oplus \Gamma(S^{-})\]
is defined by
\[\mathfrak{F}_{\omega}(A, \phi) = \left(\frac{1}{2}\rho(F_{A^t}^{+} - 4\omega^{+}) - (\phi\otimes\phi^{*})_{0}, D_{A}^{+}(\phi)\right).\]
Here, $F_{A^t}^{+}$ denotes the self-dual part of the curvature 2-form $F_{A^t}$ of the connection $A^{t}$, $\omega^{+}$ denotes the self-dual part of the perturbing 2-form $\omega$, and $(\phi\otimes\phi^{*})_{0} = \phi\otimes\phi^{*} - \frac{1}{2}|\phi|^2\text{\textbf{1}}_{S^{+}}$ is the traceless part of the endomorphism $\phi\otimes\phi^{*}: S^{+} \to S^{+}$. 

There is an action of the gauge group $\mathcal{G} := \text{Aut}(\mathfrak{s}) = \text{Map}(X, S^1)$ on $\mathcal{C}(X, \mathfrak{s})$, given by:
\[u\cdot (A, \phi) = (A - u^{-1}du, u\phi).\]
The Seiberg-Witten map $\mathfrak{F}_{\omega}$ descends to the quotient $\mathcal{B}(X, \mathfrak{s}) := \mathcal{C}(X, \mathfrak{s})/\mathcal{G}$.  The action is free on the space of irreducible configurations $\mathcal{C}^{*}(X, \mathfrak{s}) := \{(A, \phi) \in \mathcal{C}(X, \mathfrak{s}) | \phi \neq 0\}$, and the corresponding free quotient is denoted by $\mathcal{B}^{*}(X, \mathfrak{s})$.

Let $b_{2}^{+}(X)$ be the dimension of the maximal subspace $H^{+}(X; \mathbb{R}) \subset H^{2}(X; \mathbb{R})$ on which the intersection form $Q_{X}$ is positive-definite. The space of perturbing 2-forms for which there exist reducible solutions to the Seiberg-Witten equations is contained in a subspace of codimension $b_{2}^{+}(X)$. In particular, the space of 2-forms $\omega$ for which there are no reducible solutions is dense whenever $b_{2}^{+}(X) > 0$, and is path-connected whenever $b_{2}^{+}(X) > 1$. 

\subsubsection{The Moduli Space of Solutions and the Invariants}
Assuming that $b_{2}^{+}(X) > 0$, the moduli space of solutions to the Seiberg-Witten equations perturbed by a generic choice of $\omega \in \Omega^2(X; i\mathbb{R})$ is a closed, smooth, finite-dimensional manifold $\mathcal{M}(X, \mathfrak{s}, \omega)$ inside of $\mathcal{B}^{*}(X, \mathfrak{s})$.  An application of the Atiyah-Singer index theorem yields that its dimension is given by:
\[d(\mathfrak{s}) = \frac{c_1(\mathfrak{s})^2 - 2\chi(X) - 3\sigma(X)}{4}\]
where $c_1(\mathfrak{s})$ is the first Chern class of $S^{+}$.

Note that the quotient map $\mathcal{C}^{*}(X, \mathfrak{s}) \to \mathcal{B}^{*}(X, \mathfrak{s})$ is a principal $\mathcal{G}$-bundle. If we choose a basepoint $x_0 \in X$, we get an evaluation homomorphism $\mathcal{G} \to S^1$ defined by $h \mapsto h(x_0)$, which induces a principal $S^1$-bundle $P \to \mathcal{B}^{*}(X, \mathfrak{s})$. This gives us a natural sequence of cohomology classes $\mu_n \in \mathcal{B}^{*}(X, \mathfrak{s})$ to integrate over the submanifold $\mathcal{M}(X, \mathfrak{s}, \omega)$:
\[\mu_n := \begin{cases} c_1(P)^{n/2} \quad \text{ if } n \equiv 0 \mod 2 \\ 0 \quad \text{ if } n \equiv 1 \mod 2.
\end{cases}\]
where $c_{1}(P)$ denotes the first Chern class of $P$. We define the Seiberg-Witten invariant of the pair $(X, \mathfrak{s})$ to be:
\[SW_{X, \mathfrak{s}} := \langle \mu_{d(\mathfrak{s})}, [\mathcal{M}(X, \mathfrak{s}, \omega)] \rangle \in \mathbb{Z}_{2}.\]
Note that the invariant lies in $\mathbb{Z}_{2}$ since we have not oriented $\mathcal{M}(X, \mathfrak{s}, \omega)$, though it can be upgraded to a $\mathbb{Z}$-valued invariant with the appropriate care. 

When we want to emphasize the role of the auxiliary data $(g, \omega)$, we will write $SW_{X, \mathfrak{s}, g, \omega}$.  The invariants are independent of $g$ and $\omega$ when $b_{2}^{+}(X) > 1$. Indeed, when $b_{2}^{+}(X) > 1$, the space of perturbing 2-forms for which all solutions to the perturbed Seiberg-Witten invariants are irreducible is path-connected, and a path $\gamma_t$ of good perturbing 2-forms induces a cobordism $\mathcal{M}(X, \mathfrak{s}, \gamma_t)$ inside $\mathcal{B}^{*}(X)$ between $\mathcal{M}(X, \mathfrak{s}, \gamma_0)$ and $\mathcal{M}(X, \mathfrak{s}, \gamma_1)$. Thus:
\[\langle \mu_{d(\mathfrak{s})}, [\mathcal{M}(X, \mathfrak{s}, \gamma_0)]\rangle = \langle \mu_{d(\mathfrak{s})}, [\mathcal{M}(X, \mathfrak{s}, \gamma_1)]\rangle.\]

\subsubsection{Bauer-Furuta Stable Cohomotopy Refinement}

The content of this subsection and the next are needed only for the statement and proof of Theorem~\ref{baraginvol}. 

When $b_{2}^{+}(X) > 1$, the Seiberg-Witten moduli space defines not only a numerical invariant of $(X, \mathfrak{s})$ after pairing with a cohomology class, but a well-defined cobordism class inside $\mathcal{B}^{*}(X)$. If we were working inside of a finite-dimensional smooth manifold $M$ rather than the infinite-dimensional space $\mathcal{B}^{*}(X)$, the Pontrjagin-Thom construction would give a bijection between the group of codimension $n$ (framed) cobordism classes of submanifolds and the $n$th cohomotopy group of $M$:
\[\Omega^{\text{framed}}_{n}(M) \xleftrightarrow{} [M, S^n].\]
In \cite{MR2025298}, Bauer and Furuta use finite-dimensional approximation to show that the Seiberg-Witten map $\mathfrak{F}_{\omega}: \mathcal{C}(X, \mathfrak{s}) \to i\mathfrak{su}(S^{+}) \oplus \Gamma(S^{-})$ defines a class in a certain (stable) cohomotopy group. We summarize the construction here. For a detailed account of the Bauer-Furuta construction, we refer the reader to the original paper \cite{MR2025298} and to work of Baraglia-Konno \cite{MR4441598}.

Let $A_0$ be a fixed $\text{spin}^{c}$ connection, and let $\mathcal{G}_{0}$ denote the \text{based} gauge group, defined by choosing a basepoint $x_0 \in X$ and taking the kernel of the evaluation map $\mathcal{G} = \text{Map}(X, S^1) \to S^1$. Note that $A_0 + \ker d \subset \mathcal{A}$ is invariant under the action of $\mathcal{G}_{0}$, with quotient isomorphic to the Picard torus $\mathbb{T}(X) := H^1(X; i \mathbb{R})/2\pi iH^1(X; \mathbb{Z})$. In this section, we will denote $\mathbb{T}(X)$ by $B$ as it is the ``base" of the construction; in more general settings $B$ may be another finite-dimensional manifold (see \cite{MR4441598}).

Following \cite{MR2025298}, we define the \textit{monopole map} as a particular nonlinear bundle map between two infinite-rank vector bundles over $B$. Define $\mathbb{U} :=B \times \Omega^1(X)$, $\mathbb{V} := (A + \ker d) \times_{\mathcal{G}_{0}} \Gamma(S^{+})$, $\mathbb{U}' := B \times (H^1(X; \mathbb{R}) \oplus \Omega^{0}(X)/\mathbb{R}\oplus i\mathfrak{su}(S^{+}))$, and $\mathbb{V}' := (A_{0} +\ker d) \times_{\mathcal{G}_{0}}\Gamma(S^{-})$. We remark that $\mathbb{U}, \mathbb{U}'$ are real vector bundles over $B$ and $\mathbb{V}, \mathbb{V}'$ are complex vector bundles over $B$. The notation $\Omega^0(X)/\mathbb{R}$ means equivalence classes of smooth functions on $X$ up to addition of constants. We define the monopole map $f: \mathbb{U} \oplus \mathbb{V} \to \mathbb{U}' \oplus \mathbb{V}'$ as follows:
\[f(A, a, \phi) = (A, a_{\text{harm}}, d^{*}a, \mathfrak{F}(A+a, \phi) - (\rho(F_{A^t}^{+}), 0)),\]
where $\mathfrak{F} = \mathfrak{F}_{0}$, the unperturbed Seiberg-Witten map. 

One can show that the monopole map $f$ has the following properties:
\begin{itemize}
\item[(1)] $f$ is of the form $f = \ell + c$, where $\ell$ is fiberwise linear Fredholm, and $c$ is fiberwise compact;
\item[(2)] $c$ vanishes on $\mathbb{U}$ (so $f|_{\mathbb{U}} = \ell|_{\mathbb{U}}$ is linear Fredholm) and $\ell(\mathbb{U}) \subset \mathbb{U}'$;
\item[(3)] $\mathbb{U}'/\ell(\mathbb{U}) \cong B \times H^{+}(X)$.
\end{itemize}
For appropriate choices of finite-rank subbundles $U' \subset \mathbb{U}'$ and $V' \subset \mathbb{V}'$, the restriction of $f$ gives rise to a map $\tilde{f}$ between $W:=\ell^{-1}(U' \oplus V')$ and $U' \oplus V'$. We write $U:= W \cap \mathbb{U}$ and $V := W \cap \mathbb{V}$. Letting $S^{V, U}$ (resp. $S^{V', U'}$) denote the fiberwise one-point compactification of $U \oplus V$ (resp. $U' \oplus V'$), boundedness properties of the monopole map show that $\tilde{f}$ extends to a map $\hat{f}: S^{V, U} \to S^{V', U'}$. Bauer and Furuta show:
\begin{theorem} The finite-dimensional approximation $\hat{f}: S^{V, U} \to S^{V', U'}$ to the monopole map $f$ defines a class in the stable cohomotopy group $\pi^{b_{2}^{+}(X)}_{S^1, H}(B; \text{ind}(D))$, where $H$ is a Sobolev completion of $\Gamma(S^{-}) \oplus \Omega^{+}(X)$ and $\text{ind}(D)$ is the virtual index bundle of the Dirac operator $D$ (given by $D_{A}$ when restricted to the fiber over $A \in A_{0} + \ker d$). We denote this class by $BF_{X, \mathfrak{s}}$. \end{theorem}
We briefly explain the notation: the group $\pi^{n}_{G, H}(Y; \lambda)$ is the $n$th $G$-equivariant stable homotopy group of $Y$ with coefficients in the virtual bundle $\lambda$ (with $G$-universe $H$, a Hilbert space equipped with an orthogonal action of $G$), defined precisely as:
\[\pi^{n}_{G, H}(Y; \lambda) := \text{colim}_{W \subset H}[W^{+}\wedge T\lambda, W^{+} \wedge S^n]^{G}\]
where the colimit is taken over all finite-dimensional subrepresentations $W$ of $H$, $W^{+}$ denotes the one-point compactification of $W$, and $T\lambda$ denotes the Thom space of $\lambda$ (in general a spectrum). In the case at hand, $S^1$ acts on $H = \Gamma(S^{-}) \oplus \Omega^{+}(X)$ by acting as the unit complex numbers on the complex vector space $\Gamma(S^{-})$ and acting trivially on $\Omega^{+}(X)$.
\subsubsection{Cohomological Seiberg-Witten invariants from $BF_{X, \mathfrak{s}}$} Bauer and Furuta show that the class $BF_{X, \mathfrak{s}} \in \pi_{S^1, H}^{b_{2}^{+}}(\mathbb{T}(X); \text{ind}(D))$ is an invariant of the pair $(X, \mathfrak{s})$ when $b_{2}^{+} > 1$, and is strictly stronger than the Seiberg-Witten invariant in the sense that $SW_{X,\mathfrak{s}}$ can be recovered from $BF_{X, \mathfrak{s}}$, but $BF_{X, \mathfrak{s}}$ may not vanish even when $SW_{X, \mathfrak{s}}$ does.

We follow Baraglia's construction of a slightly more general cohomological form of the Seiberg-Witten invariant starting from the Bauer-Furuta set-up \cite{baraglia2023mod}. This takes the form of a map:
\[SW_{X, \mathfrak{s}}: H_{S^1}^{*}(pt) \to H^{*}(B) = H^{*}(\mathbb{T}(X)).\]
Note that this map is not necessarily grading preserving; we ignore issues of grading here. To extract this Seiberg-Witten invariant from $\hat{f}$, we first choose a section $\eta: B \to \mathbb{U}' \setminus \ell(\mathbb{U})$ (called a \textit{chamber} by Baraglia \cite{baraglia2023mod} and Baraglia-Konno \cite{MR4441598}). Using property (3) of the monopole map stated previously, homotopy classes of such sections are in correspondence with homotopy classes of nonvanishing sections $\widetilde{\eta}: B \to \mathbb{U}'/\ell(\mathbb{U}) \cong B \times H^{+}(X)$, i.e. homotopy classes of maps $B \to S(H^{+}(X))$, the unit sphere in $H^{+}(X)$. This can be interpreted as a family of ``good" perturbing 2-forms for the Seiberg-Witten equations, parametrized by the base $B = \mathbb{T}(X)$.

With properly chosen finite-rank subbundles $U', V'$, one can ensure that $U'/\ell(U) \cong \mathbb{U}'/\ell(\mathbb{U}) \cong B \times H^{+}(X)$, so we can instead take a chamber in the finite-dimensional approximation $\eta: B \to U' \setminus \ell(U)$. Following Baraglia \cite{baraglia2023mod}, one obtains the cohomological Seiberg-Witten invariant of the pair $(X, \mathfrak{s})$ in the chamber $\eta$ using a bit of algebraic topology. We summarize the procedure here: let $N^{U}$ be a tubular neighborhood of $S^{U} \subset S^{V, U}$, and consider the manifold with boundary $\widetilde{Y}^{V, U}:= S^{V, U} \setminus N^{U}$. The $S^1$-action on $S^{V, U}$ inherited from action by complex scalars on $V$ and the trivial action on $U$ is free on $\widetilde{Y}^{V, U}$, and we denote by $Y^{V, U}$ the quotient. Then we have a canonical isomorphism $\Psi:H_{S^1}^{*}(S^{V, U}, S^{U}) \cong H^{*}(Y^{V, U}, \partial Y^{V, U})$ by cohomological excision and homotopy invariance. Moreover, the projection $\pi_{V, U}: (Y^{V, U}, \partial Y^{V, U}) \to B$ gives rise to a pushforward map on ordinary cohomology $(\pi_{V, U})_{*}: H^{*}(Y^{V, U}, \partial Y^{V, U}) \to H^{*}(B)$.

Now, let $N_{\eta} \to \eta(B)$ be the normal bundle of the embedding $\eta: B \hookrightarrow U' \setminus \ell(U)$, identified with an $S^1$-invariant tubular neighborhood of $\eta(B)$. Its $S^1$-equivariant Thom class $\tau_{N_{\eta}}$ can be considered as an element of $H^{*}_{S^1}(S^{V', U'}, S^{V', U'} \setminus \eta(B))$ by excision, and then pulled back to $H_{S^1}^{*}(S^{U', V'}, S^{U'})$ since $\eta(B)$ is disjoint from $S^{U'}$ by assumption. We denote its image in the latter cohomology ring by $\tau_{\eta} \in H_{S^1}^{*}(S^{U', V'}, S^{U'})$. We can pull this back along the finite-dimensional approximation of the monopole map $\hat{f}$ to get an element $\hat{f}^{*}(\tau_{\eta}) \in H_{S^1}^{*}(S^{V, U}, S^{U})$. Baraglia's cohomological formulation of the Seiberg-Witten invariant is defined as follows:
\[SW_{X, \mathfrak{s}}^{\eta}(\theta) = (\pi_{V, U})_{*}(\Psi(\theta \hat{f}^{*}(\tau_{\eta}))).\]
Here, $\theta \in H_{S^1}^{*}(pt) \cong \mathbb{Z}_{2}[x]$, and the multiplication in the argument uses the $H_{S^1}^{*}(pt)$-module structure of $H_{S^1}^{*}(S^{V, U}, S^{U})$.
\subsubsection{$Pin(2)$ Symmetry}
When $\mathfrak{s}$ is a self-conjugate $\text{spin}^{c}$ structure, the Seiberg-Witten equations for the pair $(X, \mathfrak{s})$ admit an additional symmetry. Indeed, for any $\text{spin}^{c}$ structure $\mathfrak{s}$, there exists a map $\jmath: \mathcal{C}(X, \mathfrak{s}) \to \mathcal{C}(X, \overline{\mathfrak{s}})$, called \textit{charge conjugation}, such that $\jmath^2 = -1$. When $\mathfrak{s} = \overline{\mathfrak{s}}$, this is an automorphism of $\mathcal{C}(X, \mathfrak{s})$, and it is an involution up to gauge which preserves solutions to the Seiberg-Witten equations. Together with the existing $S^1$ symmetry of the Seiberg-Witten equations, this gives rise to a symmetry with respect to the group $Pin(2) = S^1 \cup jS^1 \subset \mathbb{H}$, where $\mathbb{H}$ denotes the quaternions and $j$ acts as $\jmath$.

In \cite{baraglia2023mod}, Baraglia constructs a version of the Seiberg-Witten invariant for pairs $(X, \mathfrak{s})$ where $\mathfrak{s}$ is self-conjugate that accounts for the full $Pin(2)$ symmetry of the Seiberg-Witten equations. This takes the form of a map:
\[SW_{X, \mathfrak{s}}^{Pin(2)}: H^{*}_{Pin(2)}(pt)\cong \mathbb{F}[V, Q]/Q^3 \to H_{\mathbb{Z}_{2}}^{*}(B).\]
His construction relies on a $Pin(2)$-equivariant version of the Bauer-Furuta set-up. It is essentially the same as the $S^1$-equivariant set-up described above, except both $U \oplus V$ and $U' \oplus V'$ are additionally equipped with involutive endomorphisms covering an involution on the base that are antilinear on the complex part and linear on the real part (this plays the role of $\jmath$). Moreover, the chamber $\eta$ needs to be chosen to be $Pin(2)$-equivariant; however, in the set-up above, there are no $Pin(2)$-invariant chambers, since $\jmath$ acts as the inversion map on the Picard torus $\mathbb{T}(X) = H^1(X; \mathbb{R})/H^1(X; \mathbb{Z})$, which has fixed points. Baraglia resolves this issue by considering the base $B$ to be the product $\mathbb{T} \times S(H^{+})$, where $S(H^{+})$ denotes the unit sphere in $H^{+}(X; \mathbb{R})$, and choosing the tautological chamber $\eta^{\text{taut}}([A], x)= x$. Here, $\jmath$ acts as the antipodal map—in particular, without fixed points—on $S(H^{+}(X; \mathbb{R}))$. Thus, the $Pin(2)$ Seiberg-Witten invariant is a map:
\[SW_{X, \mathfrak{s}}^{Pin(2)}: \mathbb{F}[V, Q]/Q^3 \to H_{\mathbb{Z}_{2}}^{*}(\mathbb{T}(X) \times S(H^{+}(X))) \cong H^{*}(\mathbb{T}(X); \mathbb{F})[Q]/Q^{b_{2}^{+}(X)}.\]

\subsection{Monopole Floer Homology}

The authoritative reference on monopole Floer homology is the book of Kronheimer and Mrowka \cite{Kronheimer_Mrowka_2007}. Let $Y$ be a smooth, closed, oriented three-manifold, and let $\mathfrak{s}$ be a $\text{spin}^{c}$-structure on $Y$. We can define the spaces $\mathcal{C}(Y, \mathfrak{s})$ and $\mathcal{B}(Y, \mathfrak{s}) := \mathcal{C}(Y, \mathfrak{s})/\mathcal{G}$ exactly as we did in the four-dimensional case. We additionally define $\mathcal{B}^{\sigma}(Y, \mathfrak{s})$, the complex blow-up of $\mathcal{B}(Y, \mathfrak{s})$ near it singular locus (see chapter 2 of \cite{Kronheimer_Mrowka_2007} for details on this construction). The blown-up space $\mathcal{B}^{\sigma}(Y, \mathfrak{s})$ has the homotopy type of $\mathbb{T}(Y) \times \mathbb{CP}^{\infty}$, where $\mathbb{T}(Y) := H^1(Y; \mathbb{R})/H^1(Y; \mathbb{Z})$ (called the \textit{Picard torus} of $Y$).

\subsubsection{The Floer Groups}

To the pair $(Y, \mathfrak{s})$, we can assign three \textit{Floer homology groups} $\widehat{\HM}(Y, \mathfrak{s})$, $\widecheck{\HM}(Y, \mathfrak{s})$, $\overline{\HM}(Y, \mathfrak{s})$ (read ``HM-from," ``HM-to," and ``HM-bar"). In order, they are analogues of the Morse homology groups $H_{*}^{\text{Morse}}(A, \partial A)$, $H_{*}^{\text{Morse}}(A)$, and $H_{*}^{\text{Morse}}(\partial A)$, where $A$ is a finite-dimensional manifold with boundary, and the Morse homology is taken with respect to a Morse function $f$ inherited from a $\mathbb{Z}/2$-invariant Morse function on the double $\widetilde{A}$ of $A$, where the action of $\mathbb{Z}/2$ is the obvious ``reflection over the boundary" (in particular, the gradient flow of $f$ is tangent to the boundary). In our case, instead of a finite-dimensional manifold, we are working on the infinite-dimensional manifold $\mathcal{B}^{\sigma}(Y, \mathfrak{s})$ whose ``boundary" is the part that maps to the reducible locus under the blow-down $\pi^{\sigma}: \mathcal{B}^{\sigma}(Y, \mathfrak{s}) \to \mathcal{B}(Y, \mathfrak{s})$, and the function in question is the Chern-Simons-Dirac functional $\mathcal{L}: \mathcal{B}^{\sigma}(Y, \mathfrak{s}) \to \mathbb{R}$, defined by:
\[\mathcal{L}(A, \phi) := -\frac{1}{8}\int_{Y} (A^t - A^{t}_{0})\wedge (F_{A^{t}} + F_{A_{0}^{t}}) + \frac{1}{2} \int_{Y} \langle D_{A}\phi, \phi \rangle d\text{vol}\]
where $A_{0}$ is a fixed reference $\text{spin}^{c}$ connection. Critical points of this functional are given precisely by solutions to the Seiberg-Witten equations on $(Y, \mathfrak{s})$:
\[\frac{1}{2}\rho(F_{A^t}) - (\phi\phi^{*})_{0} = 0,\]
\[D_{A}\phi = 0.\]
Those solutions for which $\phi = 0$ lie on the boundary of $\mathcal{B}^{\sigma}(Y, \mathfrak{s})$, and are, as in the four-dimensional case, called reducible solutions. 

As in finite-dimensional Morse homology, the critical points of $\mathcal{L}$ (after perturbation; see below) break up into three categories: (1) interior critical points, (2) boundary-stable critical points (those for which the normal direction is in the positive eigenspace of the Hessian of $f$), and (3) boundary-unstable boundary critical points (those for which the normal is in the negative eigenspace of the Hessian). We denote by $\mathfrak{C}^{o}$, $\mathfrak{C}^{s}$, and $\mathfrak{C}^{u}$ the free $\mathbb{F}$-modules generated by the interior, boundary-stable, and boundary-unstable critical points, respectively. For $i, j \in \{o, s, u\}$, we can define maps $\partial_{j}^{i}: \mathfrak{C}^{i} \to \mathfrak{C}^{j}$ and $\overline{\partial}_{j}^{i}: \mathfrak{C}^{i} \to \mathfrak{C}^{j}$ by counting downward gradient flow lines from points in $\mathfrak{C}^{i}$ to points in $\mathfrak{C}^{j}$ which lie in the irreducible part (interior) and the reducible part (boundary), respectively.  Note that $\overline{\partial}_{j}^{i}$ is only nonzero if $i, j \in \{s, u\}$. The Floer chain complexes are defined by:
\[\widehat{\CM}(Y, \mathfrak{s}) = \mathfrak{C}^{o} \oplus \mathfrak{C}^{u}, \quad \widehat{\partial} = \begin{bmatrix} \partial_{o}^{o} & \partial^{u}_{o} \\ \overline{\partial}_{u}^{s}\partial^{o}_{u} & \overline{\partial}^{u}_{u} + \overline{\partial}^{s}_{u}\partial^{u}_{s}\end{bmatrix}\]
\[\widecheck{\CM}(Y, \mathfrak{s}) = \mathfrak{C}^{o} \oplus \mathfrak{C}^{s}, \quad \widecheck{\partial} = \begin{bmatrix} \partial_{o}^{o} & \partial^{u}_{o}\overline{\partial}^{s}_{u} \\ \partial^{o}_{s} & \overline{\partial}^{s}_{s} + \partial^{u}_{s}\overline{\partial}^{s}_{u}\end{bmatrix}\]
\[\overline{\CM}(Y, \mathfrak{s}) = \mathfrak{C}^{s} \oplus \mathfrak{C}^{u}, \quad \overline{\partial} = \begin{bmatrix} \overline{\partial}^{s}_{s} & \overline{\partial}^{u}_{s} \\ \overline{\partial}^{s}_{u} & \overline{\partial}^{u}_{u}\end{bmatrix}.\]
One should think about the differentials as counting all the different possible ways that a flow line or once-broken flow line one can get from one critical point generating the complex to another.

Like the groups $H_{*}^{\text{Morse}}(A, \partial A)$, $H_{*}^{\text{Morse}}(A)$, and $H_{*}^{\text{Morse}}(\partial A)$, the Floer groups fit into a long exact sequence:
\[\la \widehat{\HM}(Y, \mathfrak{s}) \overset{p_{*}}{\la} \overline{\HM}(Y, \mathfrak{s}) \overset{i_{*}}{\la} \widecheck{\HM}(Y, \mathfrak{s}) \overset{j_{*}}{\la}\widehat{\HM}(Y, \mathfrak{s}) \la\]
We define $\HM_{\text{red}} := \ker p_{*} \cong \text{coker }j_{*}$.

The Floer groups are more than just abelian groups; they are modules over the cohomology ring of $\mathcal{B}^{\sigma}(Y, \mathfrak{s})$:
\[H^{*}(\mathcal{B}^{\sigma}(Y, \mathfrak{s}); \mathbb{F}) \cong H^{*}(\mathbb{T}(Y) \times \mathbb{CP}^{\infty}; \mathbb{F})\]
\[ \cong \Lambda^{*}(H_1(Y)/\text{torsion}) \otimes \mathbb{F}[U].\]

\subsubsection{Perturbations}

In order to make $\mathcal{L}: \mathcal{B}^{\sigma}(Y, \mathfrak{s}) \to \mathbb{R}$ into a Morse function, we need to add a small perturbation $\mathfrak{p}: \mathcal{B}^{\sigma}(Y, \mathfrak{s}) \to \mathbb{R}$. We choose $\mathfrak{p}$ from the space of \textit{tame perturbations} (see chapter 11 of \cite{Kronheimer_Mrowka_2007}).

The Floer homology groups $\HM^{\circ}(Y, \mathfrak{s})$ are \textit{natural} invariants of the pair $(Y, \mathfrak{s})$, in the sense that for any two choices of tame perturbation $\mathfrak{p}$, $\mathfrak{p}'$, there is a canonical chain homotopy class of chain homotopy equivalence
\[\Phi_{\mathfrak{p}, \mathfrak{p}'}: \CM^{\circ}(Y, \mathfrak{s}, \mathfrak{p}) \to \CM^{\circ}(Y, \mathfrak{s}, \mathfrak{p}').\]
In particular, $\Phi_{\mathfrak{p}, \mathfrak{p}'}$ induces a canonical isomorphism between the corresponding homology groups. The maps $\Phi_{\mathfrak{p}, \mathfrak{p}'}$ are often referred to as \textit{continuation maps}.

\subsubsection{Involutive Monopole Floer Homology}

Involutive monopole Floer homology was introduced by Lin in \cite{MR3968878}, inspired by work of Hendricks and Manolescu in Heegaard Floer theory \cite{MR3649355}. Its construction is analogous to that of involutive Heegaard Floer homology, in the sense that one constructs a homotopy involution $\iota$ by taking the continuation map associated to a path from one choice of auxiliary data used in the construction of the Floer homology to a ``conjugate" choice of auxiliary data. Suppose that $Y$ is equipped with a self-conjugate $\text{spin}^{c}$ structure $\mathfrak{s}$ (i.e. a $\text{spin}^{c}$ structure which is induced by a spin structure on $Y$). Just as in the four-dimensional case, we can define the charge-conjugation map $\jmath: \mathcal{C}(Y, \mathfrak{s}) \to \mathcal{C}(Y, \mathfrak{s})$ which preserves critical points of the Chern-Simons-Dirac functional. It pushes forward a tame perturbation $\mathfrak{p}$ to another tame perturbation $\jmath_{*}\mathfrak{p}$. 

On the one hand, $\jmath$ induces a canonical isomorphism (indeed, an identification of critical points) $\jmath_{*}: \CM^{\circ}(Y, \mathfrak{s}, \mathfrak{p}) \to \CM^{\circ}(Y, \mathfrak{s}, \jmath_{*}\mathfrak{p})$. On the other hand, we also have a continuation map:
\[\Phi_{\jmath_{*}\mathfrak{p}, \mathfrak{p}}: \CM^{\circ}(Y, \mathfrak{s}, \jmath_{*}\mathfrak{p}) \to \CM^{\circ}(Y, \mathfrak{s}, \mathfrak{p})\]
which is a chain homotopy equivalence. Together, these define a homotopy involution
\[\iota^{\circ} := \Phi_{\jmath_{*}\mathfrak{p}, \mathfrak{p}} \circ \jmath_{*}: \CM^{\circ}(Y, \mathfrak{s}, \mathfrak{p}) \to \CM^{\circ}(Y, \mathfrak{s}, \mathfrak{p}).\]
Then the \textit{involutive monopole Floer complex} $\CMI^{\circ}(Y, \mathfrak{s}, \mathfrak{p})$ is defined to be the mapping cone complex of $Q(1 + \iota^{\circ})$, where $Q$ is a formal variable. It is a module over the ring $\mathbb{F}[U, Q]/Q^2$.

\subsection{Computing $\overline{\CM}$ Using Coupled Morse Homology}
In chapters 33-35 of \cite{Kronheimer_Mrowka_2007}, Kronheimer and Mrowka develop a toolkit for computing $\overline{\HM}$ which we will use in section 8 for the sake of proving our main theorem (Theorem~\ref{SWstabadjintro}). This toolkit is called the \textit{coupled Morse chain complex}. Throughout this section, we omit many technical details. We refer the reader to the text of Kronheimer and Mrowka for a detailed account.

\subsubsection{The Coupled Morse Chain Complex}
Let $Q$ be a smooth, closed manifold with a Riemannian metric $g$, $H$ a separable complex infinite-dimensional Hilbert space, and $H_{1} \subset H$ a dense subspace which is also a Hilbert space equipped with its own inner product. The model for $H_{1} \subset H$ is $H^1(M; E) \subset L^2(M; E)$, the $H^1$ and $L^2$ sections of a vector bundle $E$ over a compact manifold $M$.

Define the Banach Lie group of unitary maps preserving $H_{1} \subset H$:
\[U(H:H_1) := \{A: H \to H| \text{ $A$ bounded}, A(H_1) \subset H_1, A^{*}(H_1) \subset H_1, A^{*}A = 1\}.\]
Furthermore, define $S(H:H_1)$ to be space of self-adjoint Fredholm operators $L: H_1 \to H$ of index 0. Given any $L \in S(H: H_1)$, the space $H$ admits a complete orthonormal basis of eigenvectors of $L$. We will want to work with a subspace $S_{*}(H:H_1) \subset S(H:H_1)$ of operators $L$ satisfying certain additional technical requirements, the most substantial of which is: if we label the eigenvalues of $L$ by integers in such a way that $\lambda_{i} \leq \lambda_{j}$ if $i \leq j$, then $\lambda_{i} \to \pm\infty$ as $i \to \pm \infty$.

By a \textit{family of self-adjoint Fredholm operators over $Q$}, we shall mean the following data: a smooth principal $U(H:H_1)$-bundle over $Q$, giving rise to associated Hilbert bundles $\mathcal{H}_{1} \subset \mathcal{H}$, and a bundle map $L: \mathcal{H}_{1} \to \mathcal{H}$ which when restricted to the fiber over $x \in Q$ lies in $S_{*}(\mathcal{H}_{x}: (\mathcal{H}_{1})_{x})$. In notation, we shall shorten this data by writing only the pair $(Q, L)$.

Kronheimer and Mrowka show (Corollary 33.1.7 of \cite{Kronheimer_Mrowka_2007}) that homotopy classes of families of self-adjoint Fredholm operators over $Q$ are classified by $[Q, U(\infty)]$.

Let $f: Q \to \mathbb{R}$ be a Morse function, and let $(Q, L)$ be a family of self-adjoint Fredholm operators such that the classifying map factors through $SU(\infty)$ (this is equivalent to the condition that there is no spectral flow around loops in $Q$). Then we can define the \textit{coupled Morse chain complex} $\overline{C}_{*}(Q, L)$, as follows: let $\overline{C}_{*}(Q, L)$ be freely generated over $\mathbb{F}$ by pairs $(q, i)$, where $q \in Q$ is a critical point of $f$ and $i \in \mathbb{Z}$. Due to the lack of spectral flow along loops, for each critical point $q \in Q$, we can label the eigenvalues of $L(q)$ by integers in such a way that, given any path $q(t)$ starting at $q_0$ and ending at $q_1$, the spectral flow along $q(t)$ is equal to
\[s(q_1) - s(q_0)\]
where
\[s(q) = \#\{i < 0 | \lambda_{i}(q) > 0\}.\]
Here, $\lambda_{i}(q)$ is the $i$th eigenvalue of $q$ under this labeling. We fix the labeling so the above property is satisfied, and we denote by $\phi_{i}(q)$ the eigenvector of $L(q)$ associated to $\lambda_{i}(q)$.

Now, the $(r, j)$-component of the differential of $(q, i) \in \overline{C}_{*}(Q, L)$ is given by counting pairs $(\gamma, [\phi])$, where $\gamma: \mathbb{R} \to Q$ and $[\phi]$ is the projective class of a nonvanishing section $\phi: \mathbb{R} \to \gamma^{*}\mathcal{H}_{1}$, satisfying the following equations:
\[\frac{d}{dt}\gamma + (\text{grad } f)_{\gamma(t)} = 0,\]
\[\gamma^{*}(\nabla)\phi + (L(\gamma(t))\phi)dt = 0,\]
along with the asymptotic conditions:
\[\gamma(t)\to q \quad \text{ and } \quad \phi(t) \sim c_0e^{-\lambda_i(q)t}\phi_{i}(q)\quad \text{as} \quad t \to -\infty,\]
\[\gamma(t)\to r \quad \text{ and } \quad \phi(t) \sim c_1e^{-\lambda_j(r)t}\phi_{j}(r)\quad \text{as} \quad t \to +\infty.\]
Here, $\nabla$ is a fixed connection on the underlying principal $U(H:H_1)$-bundle, and $c_0, c_1$ are positive constants. 

Kronheimer and Mrowka prove that this gradient-flow-counting map squares to zero, so $\overline{C}_{*}(Q, L)$ equipped with this map forms a chain complex. They moreover prove (Proposition 33.3.8 of \cite{Kronheimer_Mrowka_2007}) that the isomorphism class of the homology $\overline{H}_{*}(Q, L)$ of the complex (the \textit{coupled Morse homology of $(Q, L)$}) is an invariant of the homotopy class of the family of self-adjoint Fredholm operators $L$. This result relies on the construction of a \textit{continuation map} which is given by counting solutions to equations similar to those above, the only difference being that the family $L$ is replaced with a homotopy of families $L_{t}$ (i.e. the family has $t$-dependence).

\subsubsection{Application to $\overline{\HM}$}

Let $Y$ be a smooth, closed, oriented three-manifold equipped with a $\text{spin}^{c}$ structure $\mathfrak{s}$ such that $c_1(\mathfrak{s})$ is torsion (this is true, for instance, if $\mathfrak{s}$ comes from a spin structure). Then the following holds (see \cite{Kronheimer_Mrowka_2007}, Proposition 4.2.1 and Corollary 4.2.2):

\begin{proposition} There exist reducible solutions to the Seiberg-Witten equations on $(Y, \mathfrak{s})$. If $(A_0, 0) \in \mathcal{C}(Y, \mathfrak{s})$ is one such reducible solution, then the set of reducible solutions is given by $\{A_0 + a \otimes 1_{S} | a \in \Omega^1(Y; i\mathbb{R}), da = 0\}$. The quotient of the set of reducible solutions by the action of the gauge group $\mathcal{G}(Y)$ can be identified with the Picard torus $\mathbb{T}(Y) := H^1(Y; i\mathbb{R})/2\pi iH^1(Y; \mathbb{Z})$.   
\end{proposition}

The torus $\mathbb{T}(Y)$ of reducible solutions parametrizes a family of self-adjoint Fredholm operators—namely, the Dirac operators $D_{A}: L^2(S) \to L^2(S)$ associated to each solution $(A, 0)$, where $L^2(S)$ denotes the $L^2$ completion of the space of sections of the spinor bundle $S$.

Let $f: \mathbb{T}(Y) \to \mathbb{R}$ be a Morse function on $\mathbb{T}(Y)$. There is a retraction $p: \mathcal{B}(Y, \mathfrak{s}) \to \mathbb{T}(Y)$ defined by $[(A_{0} + a \otimes 1_{S}, \phi)] \mapsto [a_{\text{harm}}]$, where $a_{\text{harm}}$ is the harmonic part of $a$. The gradient of $f \circ p$ belongs to the class of tame perturbations. The reducible critical points of the perturbed action functional $\mathcal{L} + f\circ p$ are precisely the critical points of $f$ in $\mathbb{T}(Y) \subset \mathcal{B}(Y, \mathfrak{s})$.

By adding an additional tame perturbation $\mathfrak{q}$ which has gradient zero on the reducible locus, we may ensure that the reducible critical points of the Chern-Simons-Dirac functional perturbed by $\nabla (f \circ p) + \mathfrak{q}$ in the blow-up $\mathcal{B}^{\sigma}(Y, \mathfrak{s})$ are non-degenerate and the moduli spaces of flow-lines are regular. 

The torus $\mathbb{T}(Y)$ now parametrizes the family $D_{A, \mathfrak{q}}: L^2(S) \to L^2(S)$ of Dirac operators perturbed by $\mathfrak{q}$. Precisely, the operator $D_{A, \mathfrak{q}}$ is defined by:
\[D_{A, \mathfrak{q}}\phi = D_{A}\phi + \mathcal{D}_{(A, 0)}\mathfrak{q}^{1}(0, \phi),\]
where $\mathfrak{q}^1$ is the $L^2(S)$ component of the perturbation $\mathfrak{q}$, and $\mathcal{D}_{(A, 0)}$ denotes the linearization at the point $(A, 0)$ in the configuration space. We can describe the reducible critical points exactly:

\begin{proposition} The reducible critical points of the Chern-Simons-Dirac functional perturbed by $\nabla(f \circ p) + \mathfrak{q}$ on the blown-up configuration space $\mathcal{B}^{\sigma}(Y, \mathfrak{s})$ are precisely the gauge equivalence classes of configurations $(A, \phi)$ with $A \in \mathbb{T}(Y)$ a critical point of $f$ and $\phi \in L^2(S)$ an eigenvector of $D_{A, \mathfrak{q}}$ (the Dirac operator associated to $A$ and perturbed by $\mathfrak{q}$).
\end{proposition}

We may consider the coupled Morse chain complex $\overline{C}_{*}(\mathbb{T}(Y), D_{-, \mathfrak{q}})$ of the family of Dirac operators $D_{A, \mathfrak{q}}$ over $\mathbb{T}(Y)$. The above proposition tells us that the generators of the bar-flavor monopole Floer chain complex $\overline{\CM}(Y, \mathfrak{s})$ (with the choice of perturbation $\nabla(f \circ p) + \mathfrak{q}$) can be identified with the generators of $\overline{C}_{*}(\mathbb{T}(Y), D_{-, \mathfrak{q}})$. In fact, the two complexes are identical.

\begin{proposition} The monopole Floer chain complex $\overline{\CM}(Y, \mathfrak{s})$ obtained by choosing the tame perturbation $\nabla (f \circ p) + \mathfrak{q}$ is precisely given by the coupled Morse chain complex $\overline{C}_{*}(\mathbb{T}(Y), D_{-, \mathfrak{q}})$. In particular, $\overline{\HM}(Y, \mathfrak{s}) \cong \overline{H}_{*}(\mathbb{T}(Y), D_{-, \mathfrak{q}})$.
\end{proposition}

\section{Background on Heegaard Floer Theory}
\subsection{Heegaard Floer Homology}
Unlike Seiberg-Witten theory, the story of Heegaard Floer theory begins with its Floer homology theory for three-manifolds, and a bit more effort is required to see how it produces useful four-manifold invariants.

Heegaard Floer homology is an invariant for closed, oriented three-manifolds, first introduced in \cite{MR2222356}. The theory associates to a (pointed) Heegaard splitting $\mathcal{H} = (\Sigma_g, \boldsymbol{\alpha}, \boldsymbol{\beta}, z)$ of a three-manifold $Y$ an $\mathbb{F}[U]$-chain complex, denoted $\CF^{-}(\mathcal{H})$ with differential $\partial^{-}$.  More specifically, $\CF^{-}(\mathcal{H})$ is the free $\mathbb{F}[U]$-module generated by the intersection points of the two tori $\mathbb{T}_{\alpha} = \alpha_1 \times \dots \times \alpha_g$ and $\mathbb{T}_{\beta} = \beta_1 \times \dots \beta_g$ inside of the symmetric space $\text{Sym}^{g}(\Sigma_{g})$, and $\partial^{-}$ is computed by counting pseudoholomorphic strips (whenever the reduced moduli space of such strips $\widetilde{\mathcal{M}}(\phi)$ has dimension 0):
\[\partial^{-}x = \sum_{y \in \mathbb{T}_{\alpha} \cap \mathbb{T}_{\beta}} \sum_{\{\phi \in \pi_2(x, y) \mid \mu(\phi) = 1\}} \# \widetilde{\mathcal{M}}(\phi) \cdot U^{n_z(\phi)}\cdot y.\]
In the above formula, $\pi_2(x, y)$ denotes the set of homology classes of Whitney disks from $x$ to $y$, $n_z(\phi)$ counts intersection points of $\phi$ with $\{z\} \times \text{Sym}^{g-1}(\Sigma_g)$ (which is always positive, by basic properties of holomorphic curves), and the condition $\mu(\phi) = 1$ ensures that $\widetilde{\mathcal{M}}$ has dimension $0$ ($\mu$ is the \textit{Maslov index}).  The homology of this chain complex is the minus flavor of \textit{Heegaard Floer homology} $\HF^{-}(\mathcal{H})$.  Note that, if $b_1(Y) > 0$, one must first modify the Heegaard splitting to be ``admissible" in order to ensure this is well-defined, see \cite{MR2222356}.

There are several versions of the Heegaard Floer chain complex (and consequently several versions of homology): Localizing at $U$ yields $\CF^{\infty}(\mathcal{H})$; taking the cokernel of the localization map yields $\CF^{+}(\mathcal{H})$; taking the quotient $\CF^{-}(\mathcal{H})/U \cdot \CF^{-}(\mathcal{H})$ yields $\widehat{\CF}(\mathcal{H})$.

The choice of basepoint $z \in \Sigma_g$ partitions the set $\mathbb{T}_{\alpha} \cap \mathbb{T}_{\beta}$ into classes, each of which corresponds to a $\text{spin}^{c}$ structure on $Y$.  This partition is such that there can exist a Whitney disk between $x$ and $y$ only if they are in the same class, so both the chain complexes and the homology groups split according to $\text{spin}^{c}$ structures on $Y$.  For a $\text{spin}^{c}$ structure $\mathfrak{s}$, we denote the corresponding chain complex and homology group, respectively, by $\CF^{\circ}(\mathcal{H}, \mathfrak{s})$ and $\HF^{\circ}(\mathcal{H}, \mathfrak{s})$, where $\circ \in \{\widehat{\phantom{Q}}, +, -, \infty\}$.

It turns out that the isomorphism class of $\HF^{\circ}(\mathcal{H}, \mathfrak{s})$ is independent of the Heegaard splitting $\mathcal{H}$ and depends only upon the diffeomorphism class of $Y$ and the choice of $\text{spin}^{c}$ structure $\mathfrak{s}$ \cite{MR2222356}.  In fact, the construction is natural in the following sense \cite{MR4337438}:

\begin{proposition}\label{chainnaturality} For fixed $Y$, $\mathfrak{s}$, and $z$, and any two Heegaard splittings $\mathcal{H}$, $\mathcal{H}'$ compliant with those choices, there is a distinguished chain homotopy equivalence (unique up to chain homotopy)
\[\Phi(\mathcal{H}, \mathcal{H}'): \CF^{\circ}(\mathcal{H}, \mathfrak{s}) \to \CF^{\circ}(\mathcal{H}', \mathfrak{s})\]
such that, for any $\mathcal{H}, \mathcal{H}', \mathcal{H}''$,
\begin{itemize}
\item[(1)] $\Phi(\mathcal{H}, \mathcal{H}) \sim \text{Id}_{\CF^{\circ}(\mathcal{H}, \mathfrak{s})}$,
    \item[(2)]  $\Phi(\mathcal{H}', \mathcal{H}'') \circ \Phi(\mathcal{H}, \mathcal{H}') \sim \Phi(\mathcal{H}, \mathcal{H}'')$,
\end{itemize}
where $\sim$ indicates that the maps are chain homotopic.
\end{proposition}

We may therefore unambiguously write $\HF^{\circ}(Y, \mathfrak{s})$, defined as the colimit of the above transitive system (or rather, the induced system of maps on homology).

The Heegaard Floer homology groups fit into various long exact sequences.  The most important one for our purposes is the one induced by localization:
\begin{equation}\label{LES}\begin{tikzcd} \arrow{r}{\delta} & \HF^{-}(Y, \mathfrak{s}) \arrow{r}{i_{*}} & \HF^{\infty}(Y, \mathfrak{s}) \arrow{r}{\pi_{*}} & \HF^{+}(Y, \mathfrak{s}) \arrow{r}{\delta} & \phantom{q} \end{tikzcd} \end{equation}

$\HF_{\text{red}}^{-}$ is defined as the kernel of $i_{*}$ and $\HF_{\text{red}}^{+}$ is defined as the cokernel of $\pi_{*}$.  Note that $\delta$ induces an isomorphism between the two.

\subsection{Cobordism Maps}

The content of this section can be found in \cite{MR2113019}.  Let $Y_1, Y_2$ be oriented three-manifolds, and $W$ a cobordism from $Y_1$ to $Y_2$ (i.e. $W$ is a smooth four-manifold with boundary $-Y_1 \sqcup Y_2$).  Choosing a $\text{spin}^{c}$ structure $\mathfrak{s}$ on $W$ (as well as a path $\gamma$ connecting basepoints, suppressed from the notation), there is a well-defined induced map on Floer homology
\[F_{W, \mathfrak{s}}^{\circ}: \HF^{\circ}(Y_1, \mathfrak{s}|_{Y_1}) \to \HF^{\circ}(Y_2, \mathfrak{s}|_{Y_2}).\]
The map $F_{W, \mathfrak{s}}^{\circ}$ depends only on the diffeomorphism type of $W$ and the $\text{spin}^{c}$ structure $\mathfrak{s}$.

This construction is functorial in the following sense: given $W_1$ a cobordism from $Y_1$ to $Y_2$ and $W_2$ a cobordism from $Y_2$ to $Y_3$, as well as $\text{spin}^{c}$ structures $\mathfrak{s}_1 \in \text{Spin}^{c}(W_1)$ and $\mathfrak{s}_2 \in \text{Spin}^{c}(W_2)$ which agree on $Y_2$, we have:
\[F^{\circ}_{W_2, \mathfrak{s}_2} \circ F^{\circ}_{W_1, \mathfrak{s}_1} = \sum_{\{\mathfrak{s} \in \text{Spin}^{c}(W) : \mathfrak{s}|_{W_i} = \mathfrak{s}_i, i = 1, 2\}} F^{\circ}_{W, \mathfrak{s}}\]
where $W := W_1 \cup_{Y_2} W_2$.  In other words, the composition of the two cobordism maps is the sum of all the cobordism maps induced by gluing the two cobordisms together and choosing a $\text{spin}^{c}$ structure which restricts to the original $\text{spin}^{c}$ structures on each of the two cobordisms.

Using their cobordism maps, Ozsváth and Szabó show that the Heegaard Floer homology groups associated to torsion $\text{spin}^{c}$ structures admit an absolute $\mathbb{Q}$-grading.  These absolute gradings are such that the cobordism maps satisfy the following grading shift formula (see \cite{MR2113019}, Theorem 7.1):
\begin{proposition} Let $(W, \mathfrak{s})$ be a cobordism from $Y_1$ to $Y_2$ such that $\mathfrak{s}$ restricts to being torsion on both boundary components.  Then $F_{W, \mathfrak{s}}^{\circ}$ has a grading shift of
\[\frac{c_1(\mathfrak{s})^2 - 2\chi(W) - 3\sigma(W)}{4}\]
\end{proposition}

\subsubsection{The mixed invariant}

Naively, one may wish to define an invariant of closed four-manifolds by removing two four-balls, treating it as a cobordism from $S^3$ to $S^3$, and looking at the induced map on Floer homology.  Unfortunately, Proposition~\ref{boring} below (from \cite{MR2113019}), together with the fact that $\HF_{\text{red}}(S^3) = 0$, implies that all of these cobordism maps must vanish.

\begin{proposition}\label{boring} Let $W$ be a cobordism with $b_{2}^{+}(W) > 0$.  Then $F_{W, \mathfrak{s}}^{\infty} = 0$ for any $\mathfrak{s} \in \text{Spin}^{c}(W)$.
\end{proposition}

However, Ozsváth and Szabó refine their construction to develop a nontrivial invariant, called the \textit{mixed invariant} of $(X, \mathfrak{s})$ (see \cite{MR2222356}, sections 8 and 9):  Assuming $b_2^{+} > 1$, we can find a hypersurface $N \subset X$ which separates $X$ into two subcobordisms, say, $X_1$ from $S^3$ to $N$ and $X_2$ from $N$ to $S^3$, such that $b_2^{+}(X_i) > 0$ and the map $\delta: H_2(N; \mathbb{Z}) \to H_2(X;\mathbb{Z})$ from the Meier-Vietoris sequence is trivial (this latter condition ensures that the $\text{spin}^{c}$ structure on $X$ can be uniquely recovered from the $\text{spin}^{c}$ structures on $X_1$ and $X_2$).  Such an $N$ is called an \textit{admissible cut}.  

Now, the cobordism maps $F_{X_1, \mathfrak{s}|_{X_1}}^{\infty}$, $F_{X_2, \mathfrak{s}|_{X_2}}^{\infty}$ must be trivial by Proposition~\ref{boring}.  Thus, looking at the long exact sequence~\eqref{LES}, it follows that $\text{im}(F^{-}_{X_1, \mathfrak{s}|_{X_1}}) \subset \HF^{-}(N, \mathfrak{s}|_{N})$ is in fact contained in $\HF_{\text{red}}^{-}(N, \mathfrak{s}|_{N})$.  Meanwhile, the triviality of $F_{X_2, \mathfrak{s}|_{X_2}}^{\infty}$ implies that $F^{+}_{W_2, \mathfrak{s}|_{X_2}}$ descends to $\HF^{+}_{\text{red}}(N, \mathfrak{s}|_{N})$.  Thus, we may define the composition
\[\Phi_{X, \mathfrak{s}} := F^{+}_{X_2, \mathfrak{s}|_{X_2}} \circ \delta^{-1} \circ F^{-}_{X_1, \mathfrak{s}|_{X_1}}\]
to get a ``mixed map'' from $\HF^{-}(S^3)$ to $\HF^{+}(S^3)$, which is the  desired invariant.  It turns out that $\Phi_{X, \mathfrak{s}}$ is independent of the choice of admissible cut, and depends only on the diffeomorphism type of $X$ and on $\mathfrak{s}$.

\subsection{Involutive Heegaard Floer Homology}
There is a natural symmetry on Heegaard diagrams given by ``conjugating'' the Heegaard data: $\mathcal{H} = (\Sigma_g, \boldsymbol{\alpha}, \boldsymbol{\beta}, z) \mapsto \overline{\mathcal{H}}=(-\Sigma_g, \boldsymbol{\beta}, \boldsymbol{\alpha}, z)$ (which induces isomorphisms $\HF^{\circ}(Y, \mathfrak{s}) \to \HF^{\circ}(Y, \overline{\mathfrak{s}})$).  The naturality of the Heegaard Floer chain complex—in the sense of Proposition~\ref{chainnaturality}—provides a distinguished chain homotopy equivalence $\Phi(\mathcal{H}, \overline{\mathcal{H}})$ corresponding to the above change in Heegaard splitting.  This gives rise to a homotopy involution $\iota$ on the Heegaard Floer chain complex, which is used by Hendricks and Manolescu in \cite{MR3649355} to define \textit{involutive Heegaard Floer homology} $\HFI^{\circ}(Y)$.  More specifically, the involutive Heegaard Floer homology of the pair $(Y, \varpi)$, where $\varpi$ is a conjugation class of $\text{spin}^{c}$ structures on $Y$, is the homology of the mapping cone:
\[\begin{tikzcd}\CF^{\circ}(Y, \varpi) \arrow{r}{Q(1+\iota)}& Q \cdot \CF^{\circ}(Y, \varpi)[-1]\end{tikzcd}\]
where $Q$ is a formal variable with $Q^2 = 0$.  Here, we are defining $\CF^{\circ}(Y, \varpi)$ to be the sum
\[\bigoplus_{\mathfrak{s} \in \varpi} \CF^{\circ}(Y, \mathfrak{s}).\]
Note that $\varpi$ contains either one or two elements, and it contains one if and only if its sole element can be lifted to a spin structure.

In \cite{MR5032324}, Hendricks, Hom, Stoffregen, and Zemke proved that involutive Heegaard Floer homology is natural in the same way as ordinary Heegaard Floer homology (i.e. in the sense of Proposition~\ref{chainnaturality}).

\subsubsection{Involutive cobordism maps}
Using a handle-by-handle construction similar to that of Ozsváth and Szabó, Hendricks and Manolescu also construct cobordism maps in involutive Heegaard Floer homology for pairs $(W, \varpi)$, where $\varpi$ is a conjugation class of $\text{spin}^{c}$ structures on $W$ (see Proposition 4.9 in \cite{MR3649355}).  They show that these maps very similar formal properties to those from the original theory.  However, they do not show that the maps are independent of the choices made in their construction.  

Invariance was later shown for maps induced by \textit{spin} cobordisms in \cite{MR5032324}.  More precisely, Hendricks, Hom, Stoffregen, and Zemke proved that the involutive cobordism maps depend on the pair $(W, \mathfrak{s})$ along with a choice of framed path $\gamma$ between the (framed) basepoints of the two boundary components.  The dependence on the framed path $\gamma$ comes from a natural action of $\text{Diff}^{f}(Y, z)$—that is, the group of diffeomorphisms of the three-manifold $Y$ preserving the basepoint $z$ as well as the framing at $z$—on $\CFI^{\circ}(Y)$ (see Theorem 1.6 of \cite{MR5032324}).  This action is nullhomotopic on $\CFI^{\circ}(S^3)$, so we need not worry about dependence on a framed path so long as we restrict attention to closed spin four-manifolds.

\section{Construction of the Involutive Mixed Invariant}
\subsection{Construction of the Mixed Map}

We begin with an algebraic lemma.

\begin{lemma}\label{vanish1} Suppose we have the following diagram in which both squares commute, and the middle row is exact:
\[\begin{tikzcd}  & A_2 \arrow{r}{a_2} \arrow{d}{F_1} & A_3 \arrow{d}{0}  \\ B_1 \arrow{r}{b_1} \arrow{d}{0} & B_2 \arrow{r}{b_2} \arrow{d}{F_2} & B_3 \\ C_1 \arrow{r}{c_1} & C_2  & \end{tikzcd}\]
Then $F_2 \circ F_1 = 0$.
\end{lemma}
\begin{proof} The proof is a standard diagram chase.  The commutativity of the upper right square implies that the image of $F_1$ lies inside the kernel of $b_2$, which by exactness is equal to the image of $b_1$.  The commutativity of the lower left square implies that $F_2$ vanishes on the image of $b_1$, hence $F_2$ vanishes on the image of $F_1$.\end{proof}
Using this, we may prove a vanishing property for the map on $\HFI^{\infty}$ under certain topological conditions, in analogy with Lemma 8.2 of \cite{MR2113019}:

\begin{lemma}\label{vanish} Let $W$ be a smooth cobordism from $Y_1$ to $Y_2$ with $b_2^{+} > 1$, and let $\varpi$ be a conjugation class of $\text{spin}^{c}$ structures on $W$.  Then the map
\[F^{I, \infty}_{W, \mathfrak{s}}: \HFI^{\infty}(Y_1, \varpi|_{Y_1}) \to \HFI^{\infty}(Y_2, \varpi|_{Y_2})\]
vanishes (regardless of the choice of framed path).
\end{lemma}
\begin{proof} We first diagonalize the intersection form of $W$ over $\mathbb{Q}$.  Due to the condition $b_2^{+} > 1$, at least two of the diagonal entries will be positive.  After multiplying by a sufficiently large positive integer, it follows that there exists a pair of $\mathbb{Z}$-homology classes $\beta_1$ and $\beta_2$ which satisfy $\beta_i^2 > 0$ ($i = 1, 2$) and $\beta_1 \cdot \beta_2 = 0$.  Let $\Sigma_i$ be a smoothly embedded surface in $W$ representing $\beta_i$ (for $i = 1, 2$), and let $Q_i$ denote the boundary of a tubular neighborhood of $\Sigma_i$.  Note that since $\beta_1 \cdot \beta_2 = 0$, we can choose $\Sigma_1$ and $\Sigma_2$ such that they are disjoint (pairs of canceling intersection points can be surgered out, cf. proof of Lemma 8.6 in \cite{MR2113019}), and hence so are $Q_1$ and $Q_2$ by choosing small enough tubular neighborhoods.

Fix a path from $Y_1$ to $Q_1$ away from $Q_2$, and take a regular neighborhood to decompose $W$ into a cobordism $W_1$ from $Y_1$ to $Y_1 \# Q_1$ followed by a cobordism $W_2$ from $Y_1 \# Q_1$ to $Y_2$.  Note that $\Sigma_1 \subset W_1$ and $\Sigma_2 \subset W_2$, so $b_2^{+}(W_i) > 0$ for $i = 1, 2$.  By Lemma 8.2 of \cite{MR2113019}, the map $F_i := F_{W_i, \varpi|_{W_i}}^{\infty}$ vanishes for $i = 1, 2$.  By Proposition 4.9 of \cite{MR3649355}, we have the following commutative diagram:
\[
\adjustbox{scale = 0.85, center}{\begin{tikzcd} Q\HF^{\infty}(Y_1, \varpi|_{Y_1})[-1] \arrow{r}{(g_{Y_1}^{\infty})_{*}}\arrow{d}{0} & \HFI^{\infty}(Y_1, \varpi|_{Y_1}) \arrow{r}{(h_{Y_1}^{\infty})_{*}}\arrow{d}{F_{1}^{I}} & \HF^{\infty}(Y_1, \varpi|_{Y_1})\arrow{d}{0} \\ Q\HF^{\infty}(Y_1\#Q_1, \varpi|_{Y_1\#Q_1})[-1] \arrow{r}{(g_{Y_1\#Q_1}^{\infty})_{*}}\arrow{d}{0} & \HFI^{\infty}(Y_1\#Q_1, \varpi|_{Y_1\#Q_1}) \arrow{r}{(h_{Y_1\#Q_1}^{\infty})_{*}}\arrow{d}{F_{2}^{I}} & \HF^{\infty}(Y_1\#Q_1, \varpi|_{Y_1\#Q_1})\arrow{d}{0} \\ Q\HF^{\infty}(Y_2, \varpi|_{Y_2})[-1] \arrow{r}{(g_{Y_2}^{\infty})_{*}} & \HFI^{\infty}(Y_2, \varpi|_{Y_2}) \arrow{r}{(h_{Y_2}^{\infty})_{*}} & \HF^{\infty}(Y_2, \varpi|_{Y_2})
\end{tikzcd}}
\]
\\
where $F_{i}^{I} := F_{W_i, \varpi|_{W_i}}^{I, \infty}$. 
From the composition law (Proposition 4.11 of \cite{MR3649355}) and the fact that $\delta H^1(Y_1 \#Q_1; \mathbb{Z}) = 0 \subset H^2(W, \partial W; \mathbb{Z})$, it follows that
\[F^{I, \infty}_{W, \varpi} = F_{2}^{I} \circ F_{1}^{I},\]
which, by Lemma~\ref{vanish1} applied to the above diagram, must vanish.\end{proof}

In analogy with Ozsváth-Szabó's definition of an admissible cut for a cobordism, we make the following definition:

\begin{definition}\label{invcut} An involutively admissible cut of $W$ is a smoothly embedded three-manifold $N \hookrightarrow W$ which splits $W$ into two pieces $W_1$, $W_2$ each with $b_{2}^{+} > 1$, such that $\delta H^1(N ; \mathbb{Z}) \subset H^2(W, \partial W ; \mathbb{Z})$ is trivial. \end{definition}

Just as any $W$ with $b_2^{+} > 1$ admits an admissible cut \cite{MR2113019}, we have the following:
\begin{proposition}\label{exist} Any cobordism $W$ with $b_2^{+} > 3$ admits an involutively admissible cut. \end{proposition}
\begin{proof} As in the proof of Lemma~\ref{vanish} above, we diagonalize the intersection form of $W$ over $\mathbb{Q}$ to find four mutually disjoint positively self-intersecting surfaces $\Sigma_i$, where $i \in \{1, 2, 3, 4\}$.  Take $Q_i$ the boundary of a tubular neighborhood of $\Sigma_i$ such that the $Q_i$ for $i \in \{1, 2, 3, 4\}$ are also mutually disjoint.  Fix paths $\gamma_1$ from $Y_1$ to $Q_1$ and $\gamma_2$ from $Y_1$ to $Q_2$ such that $\gamma_i$ avoids $Q_j$ for all $j \neq i$, and such that $\gamma_1$ and $\gamma_2$ avoid each other.  Use $\gamma_1$ and $\gamma_2$ to form the connect sum $Y_1 \# Q_1 \# Q_2$, and in this way decompose $W$ into $W_1$ from $Y_1$ to $Y_1 \# Q_1 \# Q_2$ and $W_2$ from $Y_1 \# Q_1 \# Q_2$ to $Y_2$ (note that we are treating $Y_1$ in the sum $Y_1 \# Q_1 \# Q_2$ as a slight perturbation of $Y_1$ into the collar neighborhood $Y_1 \times [0, 1] \cong \text{nbhd}(Y_1) \subset W$).  Both $W_1$ and $W_2$ have $b_2^{+} > 1$ (namely, $W_1$ contains $\Sigma_1$ and $\Sigma_2$, while $W_2$ contains $\Sigma_3$ and $\Sigma_4$), and one can check that $\delta H^1(Y_1 \# Q_1 \# Q_2) = 0$ inside $H^2(W, \partial W)$.  Thus, $N:= Y_1 \# Q_1 \# Q_2$ is an involutively admissible cut.\end{proof}

If $W$ has an involutively admissible cut $N$, then we can construct the \textit{involutive mixed map} $F_{W, \varpi, N}^{I, \text{mixed}}: \HFI^{-}(Y_1, \varpi |_{Y_1}) \to \HFI^{+}(Y_2, \varpi |_{Y_2})$ in exactly the same way as Ozsváth-Szabó's mixed map for regular Heegaard Floer homology.  We review the construction here for the sake of completeness.  Let $W_1$ and $W_2$ be the two cobordisms into which $N$ cuts $W$.  In the following commutative diagram,
\[\adjustbox{scale = 0.9, center}{\begin{tikzcd} \arrow{r} & \HFI^{-}(Y_1, \varpi|_{Y_1}) \arrow{r}{i_{*}} \arrow{d}{F_{W_1, \varpi|_{W_1}}^{I, -}} & \HFI^{\infty}(Y_1, \varpi|_{Y_1}) \arrow{r}{j_{*}} \arrow{d}{F_{W_1, \varpi|_{W_1}}^{I, \infty}} & \HFI^{+}(Y_1, \varpi|_{Y_1}) \arrow{d}{F_{W_1, \varpi|_{W_1}}^{I, +}} \arrow{r} & \phantom{q} \\ \arrow{r} & \HFI^{-}(N, \varpi|_{N}) \arrow{r}{i_{*}} & \HFI^{\infty}(N, \varpi|_{N}) \arrow{r}{j_{*}} & \HFI^{+}(N, \varpi|_{N}) \arrow{r} & \phantom{q}  
\end{tikzcd}}\]
we have that $F^{I, \infty}_{W_1, \varpi|_{W_1}} = 0$ by Lemma~\ref{vanish}.  Thus, by commutativity of the left square, we must have that $i_{*} \circ F^{I, -}_{W_1, \varpi|_{W_1}} = 0$, so that $\text{im } F^{I, -}_{W_1, \varpi|_{W_1}} \subset \ker i_{*} = \HFI_{\text{red}}^{-}(N, \varpi|_{N})$.  Meanwhile, in the commutative diagram
\[\adjustbox{scale = 0.9, center}{\begin{tikzcd} \arrow{r} & \HFI^{-}(N, \varpi|_{N}) \arrow{r}{i_{*}} \arrow{d}{F_{W_2, \varpi|_{W_2}}^{I, -}} & \HFI^{\infty}(N, \varpi|_{N}) \arrow{r}{j_{*}} \arrow{d}{F_{W_2, \varpi|_{W_2}}^{I, \infty}} & \HFI^{+}(N, \varpi|_{N}) \arrow{d}{F_{W_2, \varpi|_{W_2}}^{I, +}} \arrow{r} & \phantom{q} \\ \arrow{r} & \HFI^{-}(Y_2, \varpi|_{Y_2}) \arrow{r}{i_{*}} & \HFI^{\infty}(Y_2, \varpi|_{Y_2}) \arrow{r}{j_{*}} & \HFI^{+}(Y_2, \varpi|_{Y_2}) \arrow{r} & \phantom{q}  
\end{tikzcd}}\]
Lemma~\ref{vanish} once again implies that $F^{I, \infty}_{W_2, \varpi|_{W_2}} = 0$.  Commutativity of the right square gives that $F^{I, +}_{W_2, \varpi|_{W_2}} \circ \pi_{*} = 0$, so $\ker \delta_{*} = \text{im } \pi_{*} \subset \ker F^{I, +}_{W_2, \varpi|_{W_2}}$.  Therefore $F^{I, +}_{W_2, \varpi|_{W_2}}$ descends to a well-defined map on $\HFI^{+}(N, \varpi|_{N})/\ker \delta_{*} = \HFI_{\text{red}}^{+}(N, \varpi|_{N})$.  Note that the boundary map $\delta$ induces an isomorphism $\overline{\delta}: \HFI_{\text{red}}^{+}(N, \varpi|_{N}) \to \HFI_{\text{red}}^{-}(N, \varpi|_{N})$, so we can define:
\[F_{W, \varpi, N}^{I, \text{mixed}} := \overline{F}^{I, +}_{W_2, \varpi|_{W_2}} \circ \overline{\delta}^{-1} \circ F^{I, -}_{W_1, \varpi|_{W_1}}.\]

This construction implicitly depends on a choice of framed path from the framed basepoint of $Y_1$ to that of $Y_2$ which passes through the basepoint of $N$ (and agrees with its framing at that point). This dependence goes away if $Y_i \cong S^3$ for $i = 1, 2$.

\subsection{Cut invariance for $b_{2}^{+}(W) > 4$}

At this point, before showing this is well-defined, we restrict ourselves to the case where $\mathfrak{s}$ is self-conjugate (i.e. spin).  It is proved in \cite{MR5032324} that the involutive cobordism maps are independent of handle decomposition for $\mathfrak{s}$ spin, so the result which follows implies that $F_{W, \mathfrak{s}}^{I, \text{mixed}}$ is a genuine invariant.  Indeed, as in the case of the original mixed map on Heegaard Floer homology, our construction is independent of the choice of cut for sufficiently large $b_2^{+}$.

\begin{proposition}\label{inv} Let $(W, \mathfrak{s})$ be a spin cobordism with $b_{2}^{+}(W) > 4$. Then the map $F_{W, \mathfrak{s}, N}^{I, \text{mixed}}$ does not depend on the choice of involutively admissible cut $N$. \end{proposition}

\begin{proof} \textbf{Case I – disjoint cuts}: Suppose first that $N$ and $N'$ are disjoint involutively admissible cuts.  Then the conclusion follows from the same argument as \cite{MR2113019}, Lemma 8.6—namely, by analyzing the commutative diagram:
\[
\adjustbox{scale=0.9, center}{\begin{tikzcd} \phantom{q} & \HFI^{+}_{\text{red}}(N, \mathfrak{s}|_{N}) \arrow{r} \arrow{d}{\overline{\delta}} & \HFI^{+}_{\text{red}}(N', \mathfrak{s}|_{N'}) \arrow{r} \arrow{d}{\overline{\delta}} & \HFI^{+}(Y_2, \mathfrak{s}|_{Y_2}) \\ \HFI^{-}(Y_1, \mathfrak{s}|_{Y_1}) \arrow{r} & \HFI^{-}_{\text{red}}(N, \mathfrak{s}|_{N}) \arrow{r} & \HFI^{-}_{\text{red}}(N', \mathfrak{s}|_{N'}) & \phantom{q}
\end{tikzcd}}\]
Here, all the horizontal maps are induced by cobordisms.
\\~\\
\textbf{Case II – $b_{2}^{+}(W) > 5$, arbitrary cuts}: Let $N$ and $N'$ be arbitrary involutively admissible cuts, with $N$ breaking $W$ into the pieces $W_1$ and $W_2$ and $N'$ breaking $W$ into the pieces $W_{1}'$ and $W_{2}'$.  Let $\Sigma_1, \Sigma_2$ be a disjoint pair of positively self-intersecting surfaces in $W_1$, and let $Q_1$, $Q_2$ denote the boundaries of their tubular neighborhoods, respectively.  Then $Y_1\#Q_1\#Q_2$ is an involutively admissible cut which, by the disjoint case above, gives the same mixed map as $N$.  Meanwhile, let $\Sigma_{1}', \Sigma_{2}'$ be a disjoint pair of positively self-intersecting surfaces in $W_{1}'$, with neighborhood boundaries $Q_{1}'$, $Q_{2}'$, respectively. Then $Y_1\#Q_{1}'\#Q_{2}'$ is an involutively admissible cut which gives the same mixed map as $N'$.  

By the assumption that $b_{2}^{+}(W) > 5$, we may choose an additional disjoint pair of positively self-intersecting surfaces $\Sigma_3, \Sigma_4 \subset W$ which are disjoint from all of the surfaces in the set $\{\Sigma_1, \Sigma_2, \Sigma_{1}', \Sigma_{2}'\}$.  Letting $Q_3$, $Q_4$ denote the boundaries of the tubular neighborhoods of $\Sigma_3$ and $\Sigma_4$, respectively, we observe that $Y_2\# Q_3\#Q_4$ gives another involutively admissible cut which is disjoint from both $Y_1\#Q_1\#Q_2$ and $Y_1\#Q_{1}'\#Q_{2}'$, and therefore yields the same mixed map as both of them.  It follows that $N$ and $N'$ both yield the same mixed map as $Y_2\# Q_3\#Q_4$, and hence the same mixed map as one another.
\\~\\
\textbf{Case III – $b_{2}^{+}(W) = 5$, arbitrary cuts}: We employ the same notation as in the first paragraph of case II.  Consider the subspace $V$ of $H_{2}(W; \mathbb{Q})$ spanned by the set $\{[\Sigma_1], [\Sigma_2], [\Sigma_{1}'], [\Sigma_{2}']\}$.  If $V$ has dimension $\leq 3$, then the assumption that $b_{2}^{+}(W) = 5$ implies that we can find classes $\alpha, \beta \in V^{\perp} := \{\gamma \in H_2(W; \mathbb{Q}) | \gamma \cdot \kappa = 0 \text{ for all } \kappa \in V\}$ such that $\alpha \cdot \beta = 0$, $\alpha^2 > 0$, and $\beta^2 > 0$.  We can then represent $\alpha$, $\beta$, respectively, by positively self-intersecting surfaces $\Sigma_3, \Sigma_4 \subset W$ which are disjoint from one another and from all of the surfaces in the set $\{\Sigma_1, \Sigma_{2}, \Sigma_{1}', \Sigma_{2}'\}$ (this is accomplished by the usual procedure of representing $\alpha$ and $\beta$ by surfaces and then tubing out intersection points).  The argument from case II then applies.

Meanwhile, if $\text{dim } V = 4$, then any two elements of the set $\{[\Sigma_1], [\Sigma_2],$ $[\Sigma_{1}'], [\Sigma_{2}']\}$ are linearly independent and so must span a subspace of dimension two.  Consider the subspace $U$ spanned by $[\Sigma_1]$ and $[\Sigma_{1}']$.  By the usual procedure, we can find a pair of disjoint embedded surfaces $\Sigma_{5}, \Sigma_{6} \subset W$ such that $[\Sigma_{5}], [\Sigma_{6}] \subset U$.  Let $Q_5$ and $Q_6$ denote boundaries of tubular neighborhoods of $\Sigma_5$ and $\Sigma_6$, respectively.  Consider the involutively admissible cut $N'' := Y_1\# Q_5\#Q_6$.  Observe that
\[\text{span}\{[\Sigma_1], [\Sigma_2], [\Sigma_5], [\Sigma_6]\} \subset \text{span}\{[\Sigma_1], [\Sigma_2], [\Sigma_{1}']\},\]
hence $\text{dim }\text{span}\{[\Sigma_1], [\Sigma_2], [\Sigma_5], [\Sigma_6]\} \leq 3$. So, the argument in the preceding paragraph (for the case $\text{dim } V \leq 3$) implies that $N$ and $N''$ induce the same mixed map.  Similarly,
\[\text{span}\{[\Sigma_{1}'], [\Sigma_{2}'], [\Sigma_5], [\Sigma_6]\} \subset \text{span}\{[\Sigma_{1}'], [\Sigma_{2}'], [\Sigma_{1}]\}\]
so that the same argument shows that $N'$ and $N''$ induce the same mixed map.  It follows that $N$ and $N'$ induce the same mixed map.\end{proof}

Note that, unfortunately, Ozsváth and Szabó's strategy of using blow-ups to prove invariance of their mixed map for $b_{2}^{+} = 2$ does not transfer over to prove invariance of our mixed map for $b_{2}^{+} =4$, since the blow-up of a spin manifold is no longer spin.  However, the proof of Proposition~\ref{inv} implies that, when $b_{2}^{+}(W) = 4$, we can assign a well-defined involutive mixed map to each two-dimensional subspace of $H_{2}(W; \mathbb{Q})$.  More precisely,

\begin{proposition}  Let $(W, \mathfrak{s})$ be a spin cobordism with $b_{2}^{+}(W) = 4$, and let $N$ (resp. $N'$) be an involutively admissible cut of $W$ splitting it into $W_1$ and $W_2$ (resp. $W_{1}'$ and $W_{2}'$).  Letting $i: W_1 \hookrightarrow W$ and $i': W_{1}' \hookrightarrow W$ denote the inclusions, suppose that $i_{*}(H_2(W_1; \mathbb{Q})) = i_{*}'(H_2(W_{1}'; \mathbb{Q})) \subset H_{2}(W, \partial W; \mathbb{Q})$.  Then $F_{W, \mathfrak{s}, N}^{\text{mix}} = F_{W, \mathfrak{s}, N'}^{\text{mix}}$.
    
\end{proposition}
\begin{proof}  Firstly, observe that invariance in the case of disjoint cuts (i.e. case I in the proof of Proposition~\ref{inv}) holds when $b_{2}^{+} = 4$.  Furthermore, observe that, since $N$ and $N'$ are involutively admissible, $i_{*}(H_2(W_1; \mathbb{Q})) = i_{*}'(H_2(W_{1}'; \mathbb{Q}))$ must be a two-dimensional subspace of $H_{2}(W, \partial W; \mathbb{Q})$.  With these observations in hand, the argument proceeds identically to the first paragraph of case III in the proof of Proposition~\ref{inv}.\end{proof}

Now, let $(X, \mathfrak{s})$ be a spin four-manifold with connected boundary $\partial X = Y$.  Assume that $b_2^{+}(X) > 4$, or $b_{2}^{+}(X) = 4$ and we have fixed a two-dimensional subspace $V$ of $H_2(X; \mathbb{Q})$.  Removing an open ball from $X$, we obtain a cobordism $\mathring{X}$ from $S^3$ to $Y$, for which there exists a well-defined mixed map:
\[F_{\mathring{X}, \mathfrak{s}}^{I, \text{mix}}: \HFI^{-}(S^3) \cong \mathbb{F}[U, Q]/Q^2 \to \HFI^{+}(Y, \mathfrak{s}|_{Y}).\]
\begin{definition} The \textup{involutive mixed invariant} of $(X, \mathfrak{s})$ is $\Phi_{X, \mathfrak{s}}^{I} := F_{\mathring{X}, \mathfrak{s}}^{I, \text{mix}}$.  For a closed four-manifold $X$, we define $\Phi_{X, \mathfrak{s}}^{I} := \Phi_{X \setminus \mathring{B^4}, \mathfrak{s}}^{I}$.
\end{definition}
\section{Properties of the Involutive Mixed Invariant}

\subsection{Relationship to the Ozsváth-Szabó Mixed Invariant}
In Proposition 4.9 of \cite{MR3649355}, Hendricks and Manolescu show that the involutive cobordism maps commute with the long exact sequence relating involutive Heegaard Floer homology to ordinary Heegaard Floer homology. The goal of this section is to prove a similar fact for the mixed invariant.
\begin{theorem}\label{mixed} Let $(X, \mathfrak{s})$ be a spin four-manifold with boundary $Y$, and let $N$ be an involutively admissible cut of $X$.  The following diagram commutes:
\[\adjustbox{scale = 0.85, center}{\begin{tikzcd} Q\HF^{-}(S^3)[-1] \cong Q\mathbb{F}[U][-1] \arrow{r}{j} \arrow{d}{\Phi_{X, \mathfrak{s}}} & \HFI^{-}(S^3) \cong \mathbb{F}[U, Q]/Q^2 \arrow{r}{p} \arrow{d}{\Phi_{X, \mathfrak{s}, N}^{I}} & \HF^{-}(S^3) \cong \mathbb{F}[U] \arrow{d}{\Phi_{X, \mathfrak{s}}}\\ Q\HF^{+}(Y, \mathfrak{s}|_{Y})[-1] \arrow{r}{g_{*}} & \HFI^{+}(Y, \mathfrak{s}|_{Y}) \arrow{r}{h_{*}} & \HF^{+}(Y, \mathfrak{s}|_{Y})  
\end{tikzcd}}\]
where $j: Q\mathbb{F}[U] \hookrightarrow \mathbb{F}[U, Q]/Q^2$ is the natural inclusion, $p: \mathbb{F}[U, Q]/Q^2 \twoheadrightarrow \mathbb{F}[U]$ is the projection given by setting $Q = 0$, $g_{*}$ is induced by inclusion $g: Q\CF^{+}(Y, \mathfrak{s}|_{Y})[-1] \hookrightarrow \CFI^{+}(Y, \mathfrak{s}|_{Y})$, and $h_{*}$ is induced by projection $h: \CFI^{+}(Y, \mathfrak{s}|_{Y}) \twoheadrightarrow \CF^{+}(Y, \mathfrak{s}|_{Y})$.
\end{theorem}

Theorem~\ref{matrix} follows as a corollary.  Since we know that the cobordism maps commute, our main focus is on showing the appropriate commutativity for the inverse of the boundary map on $\HFI_{\text{red}}$ and $\HF_{\text{red}}$.

\begin{lemma}\label{boundarycommute} The following diagram commutes for any three-manifold $Z$ and spin-structure $\mathfrak{t}$:
\[\begin{tikzcd} Q\HF^{+}(Z, \mathfrak{t})[-1] \arrow{r}{g_{*}^{+}} \arrow{d}{\delta} & \HFI^{+}(Z, \mathfrak{t}) \arrow{r}{h_{*}^{+}} \arrow{d}{\delta^{I}} & \HF^{+}(Z, \mathfrak{t}) \arrow{d}{\delta} \\ Q\HF^{-}(Z, \mathfrak{t})[-1] \arrow{r}{g_{*}^{-}} & \HFI^{-}(Z, \mathfrak{t}) \arrow{r}{h_{*}^{-}} & \HF^{-}(Z, \mathfrak{t})
\end{tikzcd}\]
where $g_{*}^{\circ}$ is induced by inclusion and $h_{*}^{\circ}$ is induced by projection.
\end{lemma}
\begin{proof} This follows from basic homological algebra applied to the following commutative diagram, in which all maps are the obvious chain maps:
\[\begin{tikzcd} Q\CF^{-}(Z, \mathfrak{t})[-1] \arrow{r} \arrow{d} & \CFI^{-}(Z, \mathfrak{t}) \arrow{r}\arrow{d} & \CF^{-}(Z, \mathfrak{t}) \arrow{d} \\ Q\CF^{\infty}(Z, \mathfrak{t})[-1] \arrow{r} \arrow{d} & \CFI^{\infty}(Z, \mathfrak{t}) \arrow{r}\arrow{d} & \CF^{\infty}(Y, \mathfrak{t}) \arrow{d} \\ Q\CF^{+}(Z, \mathfrak{t})[-1] \arrow{r} & \CFI^{+}(Z, \mathfrak{t}) \arrow{r} & \CF^{+}(Z, \mathfrak{t}) 
\end{tikzcd}\]\end{proof}

\begin{lemma}\label{descend} Let $Z$ be a three-manifold and $\mathfrak{t}$ be a spin structure on $Z$.  Consider the sequence in Floer homology
\[\begin{tikzcd} Q\HF^{\pm}(Z, \mathfrak{t})[-1] \arrow{r}{g_{*}^{\pm}} & \HFI^{\pm}(Z, \mathfrak{t}) \arrow{r}{h_{*}^{\pm}} & \HF^{\pm}(Z, \mathfrak{t}).
\end{tikzcd}\]
Then $g_{*}^{-}$ and $h_{*}^{-}$ (resp. $g_{*}^{+}$ and $h_{*}^{+}$) restrict (resp. descend) to give well-defined maps
\[g_{*}^{\pm}: Q\HF^{\pm}_{\text{red}}(Z, \mathfrak{t})[-1] \to \HFI_{\text{red}}^{\pm}(Z, \mathfrak{t}),\]
\[h_{*}^{\pm}: \HFI^{\pm}_{\text{red}}(Z, \mathfrak{t}) \to \HF_{\text{red}}^{\pm}(Z, \mathfrak{t}).\]

\end{lemma}
\begin{proof}
    We will prove the statement for $g_{*}^{\pm}$ only.  The proofs for $h_{*}^{\pm}$ are nearly identical.  We begin with $g_{*}^{-}$. The minus-flavor reduced Floer groups can be characterized as the image of the appropriate boundary operator.  So we may write any $x \in Q\HF^{-}(Z, \mathfrak{t})[-1]$ as $x = \delta y$ for some $y \in Q\HF^{+}(Z, \mathfrak{t})[-1]$.  By Lemma~\ref{boundarycommute}, we have that 
    \[g_{*}^{-}(\delta y) = \delta^{I} g_{*}^{+}(y) \in \text{im } \delta^{I} = \HFI_{\text{red}}^{-}(Z, \mathfrak{t}).\]
    This proves the claim about $g_{*}^{-}$.  Now, to show the claim for $g_{*}^{+}$, we must show that $g_{*}^{+}$ sends elements of $\ker \delta$ to elements of $\ker \delta^{I}$.  Indeed, if $x \in \ker \delta$, then, again by Lemma~\ref{boundarycommute}, we have
    \[\delta^{I} g_{*}^{+}(x) = g_{*}^{-}(\delta x) = g_{*}^{-}(0) = 0.\]
    This concludes the proof.\end{proof}

We will now prove Theorem~\ref{mixed}.
\begin{proof} We denote by $W$ the cobordism $\mathring{X}$ from $S^3$ to $Y$, by $N$ an involutively admissible cut of $W$, and by $W_i$, $i \in \{1, 2\}$, the two pieces into which $N$ cuts $W$.  We expand out the diagram stated in the theorem using the definitions of $\Phi^{I}_{X, \mathfrak{s}}$ and $\Phi_{X, \mathfrak{s}}$, with no assumption of commutativity.

\begin{equation}\begin{tikzcd} Q\HF^{-}(S^3)[-1] \arrow{r}{j} \arrow{d}{F_{W_1, \mathfrak{s}|_{W_1}}^{-}} & \HFI^{-}(S^3) \arrow{r}{p} \arrow{d}{F^{I, -}_{W_1, \mathfrak{s}|_{W_1}}} & \HF^{-}(S^3) \arrow{d}{F^{-}_{W_1, \mathfrak{s}|_{W_1}}} \\ Q\HF_{\text{red}}^{-}(N, \mathfrak{s}|_{N})[-1] \arrow{r}{k_{*}^{-}} \arrow{d}{\overline{\delta}^{-1}} & \HFI_{\text{red}}^{-}(N, \mathfrak{s}|_{N})\arrow{d}{\overline{\delta}^{-1}} \arrow{r}{\ell_{*}^{-}}& \HF_{\text{red}}^{-}(N, \mathfrak{s}|_{N})  \arrow{d}{\overline{\delta}^{-1}} \\ Q\HF_{\text{red}}^{+}(N, \mathfrak{s}|_{N})[-1] \arrow{d}{\overline{F}^{+}_{W_2, \mathfrak{s}|_{W_2}}} \arrow{r}{k_{*}^{+}}& \HFI_{\text{red}}^{+}(N, \mathfrak{s}|_{N}) \arrow{d}{\overline{F}^{I, +}_{W_2, \mathfrak{s}|_{W_2}}} \arrow{r}{\ell_{*}^{+}} & \HF_{\text{red}}^{+}(N, \mathfrak{s}|_{N}) \arrow{d}{\overline{F}^{+}_{W_2, \mathfrak{s}|_{W_2}}} \\ Q\HF^{+}(Y, \mathfrak{s}|_{Y}) \arrow{r}{g_{*}} & Q\HFI^{+}(Y, \mathfrak{s}|_{Y}) \arrow{r}{h_{*}} & \HF^{+}(Y, \mathfrak{s}|_{Y})
\end{tikzcd}\end{equation}
where the functions $k_{*}^{\pm}$ and $\ell_{*}^{\pm}$ were shown to exist in Lemma~\ref{descend}.  Now, the commutativity of the top and bottom pairs of squares follows from Proposition 4.9 of \cite{MR3649355}.  Meanwhile, the commutativity of the middle pair follows from Lemma~\ref{boundarycommute}, along with the fact that $\overline{\delta}$ is an isomorphism in each column.  The commutativity of each square implies the commutativity of the entire diagram.\end{proof}

By virtue of the surjectivity of $p$ above, Theorem~\ref{mixed} immediately implies
\begin{corollary} If $\Phi^{I}_{X, \mathfrak{s}, N} = 0$ for some involutively admissible cut $N$, then $\Phi_{X, \mathfrak{s}} = 0$.
\end{corollary}

\subsection{The invariant for closed four-manifolds}

For closed spin four-manifolds, Theorem~\ref{mixed} implies that $\Phi_{X,\mathfrak{s}}^{I}$ takes the form
\[\Phi_{X, \mathfrak{s}}^{I} = \begin{pmatrix} \Phi_{X, \mathfrak{s}} & 0 \\ \Psi_{X, \mathfrak{s}} & \Phi_{X, \mathfrak{s}} \end{pmatrix}\]
where we are treating elements of the form $U^n + QU^m \in \mathbb{F}[U, Q]/Q^2$ as vectors $(U^n, U^m)$, and $\Psi_{X, \mathfrak{s}}$ is some $U$-equivariant homomorphism which is defined by the above formula.  In other words, we have that:
\[\Phi_{X, \mathfrak{s}}^{I}(U^n + QU^m) = \Phi_{X, \mathfrak{s}}(U^n) + Q\Phi_{X, \mathfrak{s}}(U^m) + Q\Psi_{X, \mathfrak{s}}(U^n),\]
for some $\Psi_{X, \mathfrak{s}}$.  By $U$-equivariance, it is enough to know $\Phi_{X, \mathfrak{s}}(1)$ and $\Psi_{X, \mathfrak{s}}(1)$ to completely determine $\Phi_{X, \mathfrak{s}}^{I}$.  We make the following basic observation:
\begin{proposition} If $\Phi_{X, \mathfrak{s}} \neq 0$, then $\Psi_{X, \mathfrak{s}} = 0$.
\end{proposition}
\begin{proof} If they were both nonzero, $\Phi_{X, \mathfrak{s}}(1)$ and $Q\Psi_{X, \mathfrak{s}}(1)$ would lie in different towers, hence have different degree parities (and, in particular, different degrees).  But they should both have degree equal to the grading shift of $\Phi_{X, \mathfrak{s}}^{I}$, which in turn is equal to $d(\mathfrak{s})$, the dimension of the Seiberg-Witten moduli space.  Therefore, at least one of them must be zero.\end{proof}

More generally, the above argument shows that
\begin{proposition} If $X$ is closed and 
\[d(\mathfrak{s}) := \frac{c_1(\mathfrak{s})^2 - 2\chi(X) - 3\sigma(X)}{4}\]
is an even integer, then $\Phi_{X, \mathfrak{s}}^{I} = \Phi_{X, \mathfrak{s}}$, extended $Q$-equivariantly.  In particular, this is true for all closed four-manifolds $X$ for which $b_{2}^{+}(X) - b_1(X)$ is odd.
\end{proposition}

However, $\Psi_{X, \mathfrak{s}}$ is not always zero, as the stabilization formula (Theorem~\ref{stabintro}) shows. Due to the general difficulty of computing involutive cobordism maps, we have yet to find any examples of four-manifolds with nontrivial $\Psi_{X, \mathfrak{s}}$ which are not obtained through stabilization.

\subsection{Adjunction for Disjoint Pairs of Spheres}

In this subsection, we prove that the nonvanishing of $\Phi_{X, \mathfrak{s}}^{I}$ obstructs the existence of disjoint pairs of spheres with positive square (Theorem~\ref{adjintro}).  The idea is to produce an involutively admissible cut of $X$ with vanishing $\HFI_{\text{red}}$ by taking the connected sum of the boundaries of tubular neighborhoods of the two spheres.

\begingroup
\def\thetheorem{\ref{adjintro}}
\begin{theorem}Let $X$ be a smooth four-manifold with $b_{2}^{+}(X) > 4$.  Suppose that there exists a disjoint pair of smoothly embedded spheres $S_1, S_2 \subset X$ such that $[S_1]^2 > 0$ and $[S_2]^{2}> 0$.  Then $\Phi_{X, \mathfrak{s}}^{I} = 0$ for all spin structures $\mathfrak{s}$ on $X$.
\end{theorem}
\addtocounter{theorem}{-1}
\endgroup
\begin{proof} Let $Q_1$ (resp. $Q_2$) denote the boundary of a tubular neighborhood of $S_1$ (resp. $S_2$).  Note that each $Q_i$ is a lens space, and in particular an $L$-space. Let $\gamma$ be an embedded path from $Q_1$ to $Q_2$ which intersects each in exactly one point.  By taking the boundary of a neighborhood of $\gamma$ and gluing it to $Q_1$ and $Q_2$ (after removing a neighborhood of a point from each), we may form the connected sum $N:= Q_1 \#Q_2$ inside $X$.  The connected sum formula for Heegaard Floer homology (see section 6 of \cite{MR2113020}) implies that the class of $L$-spaces is closed under connected sum, so $N$ is an $L$-space.  By Corollary 4.8 of \cite{MR3649355}, $\HFI_{\text{red}}(N, \mathfrak{t}) = 0$ for any $\text{spin}^{c}$ structure $\mathfrak{t}$ on $N$.

On the other hand, we claim that $N$ is an involutively admissible cut of $X$.  On one side, $N$ bounds the union of the tubular neighborhoods of $S_1$ and $S_2$, along with a tube (a copy of $B^3 \times I$) connecting them.  We will henceforth refer to this union as $X_1$, and we will call the closure of its complement $X_2$. Since the second homology of $X_1$ is generated by $\{[S_1], [S_2]\}$, and both of these classes have positive square, we have that $b_{2}^{+}(X_1) = 2$.  But $b_{2}^{+}(X) > 4$, so $b_{2}^{+}(X_2) > 2$.  It follows that $N$ is an involutively admissible cut of $X$.

Now, given $\mathfrak{s}$ a spin structure on $X$, the map $\Phi_{X, \mathfrak{s}, N}^{I}$ factors through $\HFI_{\text{red}}^{+}(N, \mathfrak{s}|_{N}) = 0$, and is therefore zero.  By Proposition~\ref{inv}, $\Phi_{X, \mathfrak{s}, N'}$ is independent of the involutively admissible cut $N'$.  We conclude that $\Phi_{X, \mathfrak{s}}^{I} = 0$ for any spin structure $\mathfrak{s}$.\end{proof}

\section{The Involutive Mixed Invariant of Stabilized Four-Manifolds}

The involutive cobordism map for a twice punctured $S^2 \times S^2$ with its unique spin structure $\mathfrak{s}_{0}$ was computed in \cite{MR5032324}:
\begin{theorem}\label{S2} The cobordism map
\[F_{S^2 \times S^2, \mathfrak{s}_{0}}^{I, -}: \HFI^{-}(S^3) \to \HFI^{-}(S^3)\]
is multiplication by $Q$.
\end{theorem}
By looking at the diagram
\[\begin{tikzcd} \HFI^{-}(S^3) \arrow{r} \arrow{d}{F^{I, -}} & \HFI^{\infty}(S^3) \arrow{r}\arrow{d}{F^{I, \infty}} & \HFI^{+}(S^3) \arrow{d}{F^{I, +}} \\ \HFI^{-}(S^3) \arrow{r}  & \HFI^{\infty}(S^3) \arrow{r} & \HFI^{+}(S^3)\end{tikzcd}\]
and noting that the map on $\HFI^{\infty}(S^3) \cong \mathbb{F}[U, U^{-1}, Q]/Q^2$ is $U^{-1}$-equivariant as well as $U$-equivariant, we can see that Theorem~\ref{S2} still holds when we replace $-$ with $+$ or $\infty$.

Recall that the \textit{stabilization} of a four-manifold $X$ (possibly with boundary) is the (interior) connect sum $X \# (S^2 \times S^2)$. Theorem~\ref{S2} is useful for computing the involutive cobordism maps induced by stabilizations of four-manifolds.  We will use it to compute the involutive mixed invariant of four-manifolds after one stabilization.

Let $(X, \mathfrak{s})$ be a smooth, spin four-manifold with $b_{2}^{+} > 1$ and boundary $\partial X = Y$ (replace $X$ with $\mathring{X}$ if $X$ is closed).  Choose an admissible cut $N$, in the sense of \cite{MR2113019}, Definition 8.3—i.e. $N \subset X$ breaks $X$ into two parts $X_1$, $X_2$ such that $b_{2}^{+}(X_i) > 0$.  Moreover, fix a pair of maps $\mathfrak{h} = (\mathfrak{h}^{-}, \mathfrak{h}^{+})$, in the category of sets,
\[\mathfrak{h}^{-}: \HFI^{-}(N, \mathfrak{s}|_{N}) \to \HFI^{-}_{\text{red}}(N, \mathfrak{s}|_{N}),\]
\[\mathfrak{h}^{+}: \HFI^{+}_{\text{red}}(N, \mathfrak{s}|_{N}) \to \HFI^{+}(N, \mathfrak{s}|_{N})\]
such that $\mathfrak{h}^{-}$ is the identity on $\HFI^{-}_{\text{red}}(N, \mathfrak{s}|_{N})$, and $q \circ \mathfrak{h}^{+}$ is the identity on $\HFI^{+}_{\text{red}}(N, \mathfrak{s}|_{N})$, where $q: \HFI^{+}(N, \mathfrak{s}|_{N}) \to \HFI^{+}_{\text{red}}(N, \mathfrak{s}|_{N})$ is the quotient map.  We may choose the maps $\mathfrak{h}^{\pm}$ to be $U$-equivariant homomorphisms by first fixing (non-canonical) splittings of $\HFI^{\pm}(N, \mathfrak{s}|_{N})$ as $\mathbb{F}[U]$-modules:
\[\HFI^{\pm}(N, \mathfrak{s}|_{N}) \cong T^{\pm} \oplus \HFI^{\pm}_{\text{red}}(N, \mathfrak{s}|_{N}),\]
where $T^{-}$ has no $U$-torsion and $T^{+}$ is in the image of $U^N$ for all large $N$, and then taking $\mathfrak{h}^{-}$ (resp. $\mathfrak{h}^{+}$) to be the projection (resp. inclusion) with respect to this splitting.  Such splittings always exist by the classification of finitely generated modules over a PID.  Note that such a splitting need not respect the action of $Q$.

Now, we define the involutive mixed invariant of the quadruple  $(X, \mathfrak{s}, N, \mathfrak{h})$ to be:
\[\Phi_{X, \mathfrak{s}, N, \mathfrak{h}}^{I} := F_{X_2, \mathfrak{s}|_{X_2}}^{I, +} \circ \mathfrak{h}^{+} \circ \overline{\delta}^{-1} \circ \mathfrak{h}^{-} \circ F_{X_1, \mathfrak{s}|_{X_1}}^{I, -}.\]

In the case that $X$ has $b_{2}^{+} > 3$ and we choose $N$ to be an involutively admissible cut, this is precisely the involutive mixed invariant of $(X, \mathfrak{s}, N)$, regardless of the choice of $\mathfrak{h}$, which in turn is independent of $N$ if $b_{2}^{+} > 4$, by Proposition~\ref{inv}.  Moreover, if we choose $\mathfrak{h}^{\pm}$ to be $U$-equivariant, then $\Phi_{X, \mathfrak{s}, N, \mathfrak{h}}^{I}$ will also be $U$-equivariant.  However, $\Phi_{X, \mathfrak{s}, N, \mathfrak{h}}^{I}$ is in general not $Q$-equivariant.  

Notation in hand, we are now ready to state a preliminary stabilization formula.

\begin{lemma}\label{stab} Let $(X, \mathfrak{s})$ be a smooth, spin four-manifold with connected boundary $Y$. Fix a choice of $N$ and $\mathfrak{h}$ as above. Let $X'$ denote $X$ stabilized once via an interior connect sum with $S^2 \times S^2$, and let $\mathfrak{s}' := \mathfrak{s}\#\mathfrak{s}_{0}$. Assume that the gluing for the stabilization occurs in the interior of $X_1$ (the half which doesn't contain the boundary $Y$).  Then:
\[\Phi_{X', \mathfrak{s}, N, \mathfrak{h}}^{I} = \Phi_{X, \mathfrak{s}, N, \mathfrak{h}}^{I} \circ Q\]
where $Q$ denotes multiplication by $Q$. 
 Moreover, if $X$ is closed and the gluing of $S^2 \times S^2$ occurs in the interior of $X_2$, then
\[\Phi_{X', \mathfrak{s}, N, \mathfrak{h}}^{I} = Q\cdot \Phi_{X, \mathfrak{s}, N, \mathfrak{h}}^{I}.\]
    
\end{lemma}
\begin{proof} We begin with the first statement.  By definition, we have that

\[\Phi_{X', \mathfrak{s}, N, \mathfrak{h}}^{I} = F_{X_2, \mathfrak{s}|_{X_2}}^{I, +} \circ \mathfrak{h}^{+} \circ \overline{\delta}^{-1}\circ \mathfrak{h}^{-} \circ F_{X_1 \# (S^2 \times S^2), \mathfrak{s}'|_{X_1 \# (S^2 \times S^2)}}^{I, -}.\]

By the composition law for cobordism maps, along with Theorem~\ref{S2}, we have that:
\[F_{X_1 \# (S^2 \times S^2), \mathfrak{s}'|_{X_1 \# (S^2 \times S^2)}}^{I, -} = F_{X_1, \mathfrak{s}|_{X_1}}^{I, -} \circ F_{S^2 \times S^2, \mathfrak{s}_{0}}^{I, -} \]
\[ = F_{X_1, \mathfrak{s}|_{X_1}}^{I, -} \circ Q.\]
Plugging this into the formula for $\Phi_{X', \mathfrak{s}, N, \mathfrak{h}}$ above completes the proof of the first statement.

The proof of the second statement is nearly identical, after observing that
\[F_{X_2\#(S^2 \times S^2), \mathfrak{s}'|_{X_2\#(S^2 \times S^2)}}^{I, +} = Q\cdot F_{X_2, \mathfrak{s}|_{X_2}}^{I, +}. \]\end{proof}

\begin{proposition}\label{Qinv} Let $(X, \mathfrak{s})$ be a closed, spin four-manifold with $b_{2}^{+}(X) \geq 3$.  Fix a choice of $N$ and $\mathfrak{h}$ as above, with the additional condition that $b_{2}^{+}(X_1) > 1$.  Then $Q\cdot \Phi_{X, \mathfrak{s}, N, \mathfrak{h}}^{I}$ is independent of $\mathfrak{h}$.  In fact, $Q \cdot \Phi_{X, \mathfrak{s}, N, \mathfrak{h}}^{I} = Q\cdot \Phi_{X, \mathfrak{s}}$, where $\Phi_{X, \mathfrak{s}}: \HFI^{-}(S^3) \to \HFI^{+}(S^3)$ denotes the Oszváth-Szabó mixed invariant extended $Q$-equivariantly.
    
\end{proposition}
\begin{proof} Let $X' = X \#(S^2 \times S^2)$, where the gluing of $S^2 \times S^2$ occurred inside $X_2$.  Note that $b_{2}^{+}(S^2 \times S^2) = 1$, so $b_{2}^{+}(X_{2}') = b_{2}^{+}(X_2) + 1 > 1$.  Thus, $N$ is an involutively admissible cut of $X'$.  By definition, we have that:
\[\Phi_{X', \mathfrak{s}, N}^{I} = F_{X_2 \# (S^2 \times S^2), \mathfrak{s}'|_{X_2 \#(S^2 \times S^2)}}^{I, +} \circ \overline{\delta}^{-1} \circ F_{X_1, \mathfrak{s}|_{X_1}}^{I, -}.\]
Since $N$ is an involutively admissible cut, $\Phi_{X', \mathfrak{s}, N}^{I} = \Phi_{X', \mathfrak{s}, N, \mathfrak{h}}^{I}$ for any choice of $\mathfrak{h}$.  Applying Lemma~\ref{stab}, we obtain that:
\[\Phi_{X', \mathfrak{s}, N}^{I} = \Phi_{X', \mathfrak{s}, N, \mathfrak{h}}^{I} = Q \cdot \Phi_{X, \mathfrak{s}, N, \mathfrak{h}}^{I}.\]
The left-hand side is independent of $\mathfrak{h}$, hence so is the right-hand side.  

It also follows that the right-hand side is a $\mathbb{F}[U, Q]/Q^2$-homomorphism, regardless of the choice of $\mathfrak{h}$.  Therefore, to show that $Q\cdot \Phi_{X, \mathfrak{s}, N, \mathfrak{h}}^{I} = Q\cdot \Phi_{X, \mathfrak{s}}$, it suffices to show that they agree at $1 \in \mathbb{F}[U, Q]/Q^2 \cong \HFI^{-}(S^3)$. 

Observe that, since $b_{2}^{+}(X_1) > 1$, Lemma~\ref{vanish} implies that the image of $F_{X_1, \mathfrak{s}|_{X_1}}^{I, -}$ is contained in $\HFI^{-}_{\text{red}}(N, \mathfrak{s}|_{N})$.  Similarly, Lemma 8.2 of \cite{MR2031164} implies that the image of $F_{X_1, \mathfrak{s}|_{X_1}}^{-}$ is contained in $\HF^{-}_{\text{red}}(N, \mathfrak{s}|_{N})$.  We consider the commutative diagram:
\[\begin{tikzcd}
    \HFI^{-}(S^3) \arrow{r} \arrow{d}{F_{X_1, \mathfrak{s}|_{X_1}}^{I, -}} & \HF^{-}(S^3) \arrow{d}{F_{X_1, \mathfrak{s}|_{X_1}}^{-}} \\ \HFI^{-}_{\text{red}}(N, \mathfrak{s}|_{N}) \arrow{r}\arrow{d}{\overline{\delta}^{-1}} & \HF^{-}_{\text{red}}(N, \mathfrak{s}|_{N}) \arrow{d}{\overline{\delta}^{-1}} \\ \HFI^{+}_{\text{red}}(N, \mathfrak{s}|_{N}) \arrow{r}{\overline{h_{N}}} & \HF^{+}_{\text{red}}(N, \mathfrak{s}|_{N})
\end{tikzcd}\]
Note that the map $\overline{h_{N}}$ and the middle horizontal map, both induced by projection, were shown to be well-defined in Lemma~\ref{descend}.  Let $A = \overline{\delta}^{-1} \circ F_{X_1, \mathfrak{s}|_{X_1}}^{I, -}(1)$, and let $a = \overline{h_{N}}(A)$.  By commutativity of the diagram, we know that $a = \overline{\delta}^{-1} \circ F_{X_1, \mathfrak{s}|_{X_1}}^{-}(1)$.  In other words, $A$ is the image of $1 \in \HFI^{-}(S^3)$ in the bottom left of the diagram, and $a$ is the image of $1$ in the bottom right of the diagram.  

Recall that we have a section $\mathfrak{h}^{+}: \HFI^{+}_{\text{red}}(N, \mathfrak{s}|_{N}) \to \HFI^{+}(N, \mathfrak{s}|_{N})$ of the projection map.  We consider the element $\mathfrak{h}^{+}(A) \in \HFI^{+}(N, \mathfrak{s}|_{N})$, and trace it through the following commutative diagram:
\begin{equation}\label{stabdiag} \begin{tikzcd}
    \HFI^{+}(N, \mathfrak{s}|_{N}) \arrow{r}{h_{N}} \arrow{d}{F_{X_2, \mathfrak{s}|_{X_2}}^{I, +}} & \HF^{+}(N, \mathfrak{s}|_{N}) \arrow{d}{F_{X_2, \mathfrak{s}|_{X_2}}^{+}} \\ \HFI^{+}(S^3) \arrow{r} & \HF^{+}(S^3)\end{tikzcd} \end{equation}
Firstly, since $\mathfrak{h}^{+}(A)$ represents the class $A$, and $\overline{h_{N}}$ is well-defined, $h_{N}(\mathfrak{h}^{+}(A))$ represents the class $a = \overline{h_{N}}(A)$.  Since $N$ is an ordinary admissible cut (in the sense of \cite{MR2031164}, Definition 8.3), the map $F_{X_2, \mathfrak{s}|_{X_2}}^{+}$ descends to a well-defined map $\overline{F}_{X_2, \mathfrak{s}|_{X_2}}: \HF^{+}_{\text{red}}(N, \mathfrak{s}|_{N}) \to \HF^{+}(S^3)$.  Thus, we can compute $\mathfrak{h}^{+}(A)$ going right and down in~\eqref{stabdiag} as follows:
\[(F_{X_2, \mathfrak{s}|_{X_2}}^{+} \circ h_{N})(\mathfrak{h}^{+}(A)) = \overline{F}_{X_2, \mathfrak{s}|_{X_2}}(a)\]
\[ = \overline{F}_{X_2, \mathfrak{s}|_{X_2}} \circ \overline{\delta}^{-1} \circ F_{X_1, \mathfrak{s}|_{X_1}}^{-}(1) = \Phi_{X, \mathfrak{s}}(1).\]
In the last step, we have used the definition of the Ozsváth-Szabó mixed invariant $\Phi_{X, \mathfrak{s}}$, along with the fact that $N$ is an ordinary admissible cut.

On the other hand, going straight down in~\eqref{stabdiag}, we obtain the element
\[F_{X_2, \mathfrak{s}|_{X_2}}^{I, +}(\mathfrak{h}^{+}(A)) = F_{X_2, \mathfrak{s}|_{X_2}}^{I, +}\circ \mathfrak{h}^{+} \circ \overline{\delta}^{-1} \circ F_{X_1, \mathfrak{s}|_{X_1}}^{I, -} (1)\]
\[ = \Phi^{I}_{X, \mathfrak{s}, N, \mathfrak{h}}(1),\]
where we recall that $\mathfrak{h} = (\mathfrak{h}^{-}, \mathfrak{h}^{+})$ for some choice of $\mathfrak{h}^{-}: \HFI^{-}(N, \mathfrak{s}|_{N}) \to \HFI^{-}_{\text{red}}(N, \mathfrak{s}|_{N})$ which restricts to the identity on the reduced group (the choice of $\mathfrak{h}^{-}$ does not matter, since $F_{X_1, \mathfrak{s}|_{X_1}}^{I, -}$ lands in $\HFI^{-}_{\text{red}}(N, \mathfrak{s}|_{N})$ anyway).

Now, the commutativity of~\eqref{stabdiag} implies that the component of $\Phi^{I}_{X, \mathfrak{s}, N, \mathfrak{h}}(1)$ which lies in the non-$Q$ tower of $\HFI^{+}(S^3) \cong \mathbb{F}[U^{-1}, Q]/Q^2$ is equal to $\Phi_{X, \mathfrak{s}}(1)$, identified as an element of $\mathbb{F}[U^{-1}] \subset \mathbb{F}[U^{-1}, Q]/Q^2$.  Since the $Q$-tower is annihilated by $Q$, it follows that:
\[Q\cdot \Phi_{X, \mathfrak{s}, N, \mathfrak{h}}^{I}(1) = Q\cdot \Phi_{X, \mathfrak{s}}(1).\]\end{proof}

A similar proof, only invoking the first statement in Lemma~\ref{stab} rather than the second, yields the following:
\begin{proposition}\label{Qinv2} Let $(X, \mathfrak{s})$ be a closed, spin four-manifold with $b_{2}^{+}(X) \geq 3$.  Fix a choice of $N$ and $\mathfrak{h}$ as above, with the additional condition that $b_{2}^{+}(X_2) > 1$.  Then $\Phi_{X, \mathfrak{s}, N, \mathfrak{h}}^{I}$ restricted to the $Q$-tower of $\HFI^{-}(S^3) \cong \mathbb{F}[U, Q]/Q^2$ is independent of $\mathfrak{h}$.  In fact, restricted to this tower, $\Phi_{X, \mathfrak{s}, N, \mathfrak{h}}^{I}$ agrees with $\Phi_{X, \mathfrak{s}}$ (extended $Q$-equivariantly).
\end{proposition}

\begin{proof} The proof is similar to that of Proposition~\ref{Qinv}.  Let $X' = X \#(S^2 \times S^2)$, where the gluing of $S^2 \times S^2$ occurred inside $X_1$.  Note that $b_{2}^{+}(S^2 \times S^2) = 1$, so $b_{2}^{+}(X_{1}') = b_{2}^{+}(X_1) + 1 > 1$.  Thus, $N$ is an involutively admissible cut of $X'$.  By definition, we have that:
\[\Phi_{X', \mathfrak{s}, N}^{I} = F_{X_2, \mathfrak{s}|_{X_2}}^{I, +} \circ \overline{\delta}^{-1} \circ F_{X_1 \# (S^2 \times S^2), \mathfrak{s}'|_{X_1 \# (S^2 \times S^2)}}^{I, -}.\]

Since $N$ is an involutively admissible cut, $\Phi_{X', \mathfrak{s}, N}^{I} = \Phi_{X', \mathfrak{s}, N, \mathfrak{h}}^{I}$ for any choice of $\mathfrak{h}$.  Applying Lemma~\ref{stab}, we obtain that:
\[\Phi_{X', \mathfrak{s}, N}^{I} = \Phi_{X', \mathfrak{s}, N, \mathfrak{h}}^{I} = \Phi_{X, \mathfrak{s}, N, \mathfrak{h}}^{I} \circ Q.\]
The left-hand side is independent of $\mathfrak{h}$, hence so is the right-hand side.  It also follows that the right-hand side is $\mathbb{F}[U, Q]/Q^2$-equivariant.

For the statement that $\Phi_{X, \mathfrak{s}, N, \mathfrak{h}}^{I}$ agrees with $\Phi_{X, \mathfrak{s}}$ on the $Q$-tower, we only sketch the proof—the details are similar to (but easier than) those in the proof of Proposition~\ref{Qinv}.  Firstly, we note that, by $U$-equivariance, they need only agree on the element $Q \in \mathbb{F}[U, Q]/Q^2 \cong \HFI^{-}(S^3)$.  

Now, since $N$ is an ordinary admissible cut, the image of $F_{X_1, \mathfrak{s}|_{X_1}}^{-}$ is contained in $\HF^{-}_{\text{red}}(N, \mathfrak{s}|_{N})$.  Consider the following commutative diagram:
\[\begin{tikzcd} Q\HF^{-}(S^3)[-1] \arrow{r} \arrow{d}{F_{X_1, \mathfrak{s}|_{X_1}}^{-}} & \HFI^{-}(S^3) \arrow{d}{F_{X_1, \mathfrak{s}|_{X_1}}^{I, -}} \\ Q\HF_{\text{red}}^{-}(N, \mathfrak{s}|_{N})[-1] \arrow{r}{g_{N}} & \HFI^{-}(N, \mathfrak{s}|_{N})
\end{tikzcd}\]

By Lemma~\ref{descend}, $g_{N}$ takes $Q\HF_{\text{red}}^{-}(N, \mathfrak{s}|_{N})[-1]$ into $\HFI^{-}_{\text{red}}(N, \mathfrak{s}|_{N})$.  Thus, tracking the element $Q \in Q\mathbb{F}[U][-1] \cong Q\HF^{-}(S^3)[-1]$ through the diagram, we obtain an element $A \in \HFI_{\text{red}}^{-}(N, \mathfrak{s}|_{N})$.  We denote by $a \in Q\HF_{\text{red}}^{-}(N, \mathfrak{s}|_{N})[-1]$ the image of $Q$ under the cobordism map $F_{X_1, \mathfrak{s}|_{X_1}}^{-}$, and we note that $g_N(a) = A$.  Now, we track $a$ through the following commutative diagram:
\[\begin{tikzcd} Q\HF^{-}_{\text{red}}(N, \mathfrak{s}|_{N})[-1] \arrow{r}{g_N} \arrow{d}{\overline{\delta}^{-1}} & \HFI^{-}_{\text{red}}(N, \mathfrak{s}|_{N}) \arrow{d}{\overline{\delta}^{-1}} \\ Q\HF^{+}_{\text{red}}(N, \mathfrak{s}|_{N})[-1] \arrow{r} \arrow{d}{F_{X_2, \mathfrak{s}|_{X_2}}^{+}} & \HFI^{+}_{\text{red}}(N, \mathfrak{s}|_{N}) \arrow{d}{F_{X_2, \mathfrak{s}|_{X_2}}^{I, +}} \\ Q\HF^{+}(S^3)[-1] \arrow{r} & \HFI^{+}(S^3)
\end{tikzcd}
\]
One can check that, tracking $a$ down then right, we obtain $Q \cdot \Phi_{X, \mathfrak{s}}(1) = \Phi_{X, \mathfrak{s}}(Q)$.  Meanwhile, tracking $a$ right and then down (i.e. tracking $A$ down), and observing that $\mathfrak{h}^{-}(A) = A$ since $A \in \HFI^{-}_{\text{red}}(N, \mathfrak{s}|_{N})$, we obtain $\Phi_{X, \mathfrak{s}, N, \mathfrak{h}}^{I}(Q)$.  The commutativity of the diagram implies that $\Phi_{X, \mathfrak{s}}(Q) = \Phi_{X, \mathfrak{s}, N, \mathfrak{h}}^{I}(Q)$, as desired.\end{proof}

We may now prove Theorem~\ref{stabintro}.

\begingroup
\def\thetheorem{\ref{stabintro}}
\begin{theorem} Let $(X, \mathfrak{s})$ be a closed spin four-manifold with $b_{2}^{+}(X) > 3$.  Then:
\[\Phi_{X \# (S^2 \times S^2), \mathfrak{s}\#\mathfrak{s}_{0}}^{I} = Q\Phi_{X, \mathfrak{s}}.\]
If $b_{2}^{+}(X) = 3$ and $N$ is an involutively admissible cut of $X\#(S^2 \times S^2)$ which is contained entirely in $X$, then the same formula holds for $\Phi_{X\#(S^2 \times S^2), \mathfrak{s}\#\mathfrak{s}_{0}, N}^{I}$.
\end{theorem}
\addtocounter{theorem}{-1}
\endgroup
\begin{proof} Firstly, observe that if $b_{2}^{+}(X) > 3$, then $X$ has an involutively admissible cut $N$, by Proposition~\ref{exist}.  By perturbing slightly, $N$ can be chosen away from the point removed from $X$ to form the connect sum $X \# (S^2 \times S^2)$, so $N$ is therefore also an involutively admissible cut of $X\#(S^2 \times S^2)$.

We henceforth assume that $b_{2}^{+}(X) \geq 3$, and we have fixed some involutively admissible cut $N \subset X \setminus \{\text{pt.}\}$ of $X \#(S^2 \times S^2)$.  Let $X_1$, $X_2$ denote the two components into which $N$ splits $X$.  By the fact that $N \subset X \setminus \{\text{pt.}\}$, the gluing of $S^2 \times S^2$ must have occurred in the interior of either $X_1$ or $X_2$.

If the gluing occurred in $X_2$, then $b_{2}^{+}(X_1) > 1$ and $b_{2}^{+}(X_2) > 0$ by the assumption that $N$ is involutively admissible for $X\#(S^2 \times S^2)$.  In particular, $N$ is an ordinary admissible cut of $X$.  Fixing a choice of $\mathfrak{h} = (\mathfrak{h}^{+}, \mathfrak{h}^{-})$, Lemma~\ref{stab} implies that
\[\Phi^{I}_{X\#(S^2 \times S^2), \mathfrak{s}\#\mathfrak{s}_{0}, N, \mathfrak{h}} = Q\cdot \Phi^{I}_{X, \mathfrak{s}, N, \mathfrak{h}}.\]

By Proposition~\ref{Qinv}, the right side is equal to $Q\Phi_{X, \mathfrak{s}}$.  The left side, meanwhile, is independent of $\mathfrak{h}$ since $N$ is an involutively admissible cut of $X\#(S^2 \times S^2)$.  Thus, we obtain that
\[\Phi^{I}_{X\#(S^2 \times S^2), \mathfrak{s}\#\mathfrak{s}_{0}, N} = Q\Phi_{X, \mathfrak{s}}.\]
If $b_{2}^{+}(X) > 3$, then $b_{2}^{+}(X\#(S^2 \times S^2))> 4$, so Proposition~\ref{inv} implies that the left side is independent of $N$.  This completes the proof in the case where the gluing of $S^2 \times S^2$ occurred in $X_2$.

In the case where the gluing of $S^2 \times S^2$ occurred in $X_1$, the proof is nearly identical, except that we invoke Proposition~\ref{Qinv2} instead of Proposition~\ref{Qinv}.\end{proof}
In conjunction with Theorem~\ref{adjintro} applied to $X\#(S^2 \times S^2)$, Theorem~\ref{stabintro} yields:
\begingroup
\def\thetheorem{\ref{stabadjintro}}
\begin{theorem} Let $X$ be a smooth, closed four-manifold with $b_{2}^{+}(X) > 3$.  Suppose that there exists a disjoint pair of smoothly embedded spheres $S_1, S_2 \subset X\#(S^2 \times S^2)$ such that $[S_{1}]^2 > 0$ and $[S_{2}]^2 > 0$.  Then $\Phi_{X, \mathfrak{s}} = 0$ for all spin structures $\mathfrak{s}$ on $X$, where $\Phi_{X, \mathfrak{s}}$ is the Ozsváth-Szabó mixed invariant.
\end{theorem}
\addtocounter{theorem}{-1}
\endgroup
\begin{proof} Theorem~\ref{adjintro} implies that $\Phi_{X, \mathfrak{s}}^{I} = 0$ for all spin structures $\mathfrak{s}$. On the other hand, Theorem~\ref{stabintro} implies that $\Phi_{X, \mathfrak{s}}^{I} = Q\Phi_{X, \mathfrak{s}}$. It follows that $\Phi_{X, \mathfrak{s}} = 0$.\end{proof}

\section{Construction of the Involutive Seiberg-Witten Invariant}

All of the results and arguments in section 4.1 remain valid when $\HF^{\circ}$ and $\HFI^{\circ}$ (where $\circ = -, +, \infty$) are replaced with $\HM^{\circ}$ and $\HMI^{\circ}$ (where $\circ = \widehat{\phantom{hm}}, \widecheck{\phantom{hm}}, \overline{\phantom{hm}}$), respectively. The only property used there which relies substantively on the nature of the Floer groups (beyond merely their formal properties) is the fact that $F_{W, \mathfrak{s}}^{\infty}$ vanishes for a cobordism $W$ with $b_{2}^{+} > 0$ (Lemma 8.2 of \cite{MR2113019}). This can be replaced directly with the fact that $\overline{\HM}(W, \mathfrak{s})$ vanishes whenever $W$ has $b_{2}^{+} > 0$ (Prosition 3.5.2 of \cite{Kronheimer_Mrowka_2007}). This should be viewed as a consequence of the fact that when $b_{2}^{+} > 0$, a generic perturbation eliminates all reducible solutions of the Seiberg-Witten equations.

As such, for a closed four-manifold $X$ equipped with a spin structure $\mathfrak{s}$ (or self-conjugate $\text{spin}^{c}$ structure) and having $b_{2}^{+} \geq 4$, we may define an \textit{involutive monopole invariant}
\[\mathfrak{m}_{X, \mathfrak{s}}^{I}: \widehat{\HMI}(S^3) \cong \mathbb{F}[U, Q]/Q^2 \to \widecheck{\HMI}(S^3) \cong \mathbb{F}[U^{-1}, Q]/Q^2\]
in exactly the same way that we defined the involutive mixed invariant $\Phi_{X, \mathfrak{s}}^{I}$, only using involutive monopole Floer homology in place of its Heegaard Floer counterpart. Though the proof of cut invariance provided in section 4.2 suffices to show cut invariance for $\mathfrak{m}_{X. \mathfrak{s}}^{I}$ when $b_{2}^{+}(X) > 4$, we will provide an entirely different proof of this fact below which extends to the case $b_{2}^{+} = 4$ (Theorem~\ref{SWequiv1}), and so we need not assume the results of 4.2.

In this section, for the sake of both providing conceptual clarity for our involutive invariants and strengthening our adjunction results, we will show that $\mathfrak{m}_{X, \mathfrak{s}}^{I}$ is closely related to several other invariants which arise naturally in Seiberg-Witten theory, including Baraglia's Pin(2) Seiberg-Witten invariant $SW_{X, \mathfrak{s}}^{Pin(2)}$ (see \cite{baraglia2023mod}).

\subsection{Construction as a count of irreducible solutions}
Chapter 27 of \cite{Kronheimer_Mrowka_2007} shows that the standard Seiberg-Witten invariant $SW_{X, \mathfrak{s}}$ can be obtained from certain cobordism maps on Floer homology: one which counts irreducible solutions which are asymptotic to reducibles on the spherical ends—denoted by $\overrightarrow{\HM}(X, \mathfrak{s})$—and one which is constructed as a mixed map (exactly like the Ozsváth-Szabó mixed map, only using monopole Floer homology). More specifically, it is shown via a composition rule for $\overrightarrow{\HM}(X, \mathfrak{s})$ that these two maps are equivalent, and both recover $SW_{X, \mathfrak{s}}$ by evaluating on the element $1 \in \widehat{\HM}(S^3) \cong \mathbb{F}[U]$.

Given that the involutive monopole invariant $\mathfrak{m}^{I}_{X, \mathfrak{s}}$ is defined as a mixed map, it is natural to consider whether it can be interpreted as a suitable count of irreducible solutions to the Seiberg-Witten equations, or more generally as a cohomological pairing against the fundamental class of a certain moduli space of solutions. We will show that this is indeed the case. 

\subsubsection{An involutive Seiberg-Witten invariant} Let $X$ be a closed four-manifold equipped with a spin structure $\mathfrak{s}$ and a fixed Riemannian metric $g$. Suppose that $b_{2}^{+}(X) \geq 3$. We choose a perturbing 2-form $\omega \in \Omega^{2}(X; i\mathbb{R})$ such that the moduli space $\mathcal{M}(X, \mathfrak{s})_{\omega}$ of solutions to the perturbed Seiberg-Witten equation $\mathfrak{F}_{\omega}(A, \Phi) = 0$ is regular and contains no reducibles (hence is a smooth, closed manifold). Moreover, we choose a path $\gamma: [0, 1] \to \Omega^{2}(X; i\mathbb{R})$ such that $\gamma(0) = \omega$, $\gamma(1) = -\omega$, and the moduli space of solutions along the path $\mathcal{M}(X, \mathfrak{s})_{\gamma} := \bigcup_{t \in [0, 1]} \{t\} \times \mathcal{M}(X, \mathfrak{s})_{\gamma(t)}$ is regular (in the sense of definition 24.4.9 in \cite{Kronheimer_Mrowka_2007}) and contains no reducible solutions (which is possible since $b_{2}^{+}(X) \geq 3$). Let $d = d(\mathfrak{s})$ denote the dimension of $\mathcal{M}(X, \mathfrak{s})_{\omega}$.

\begin{definition} The \textit{involutive Seiberg-Witten invariant}, denoted $SW_{X, \mathfrak{s}}^{I}$, is defined as follows: If $d$ is even, then
\[SW_{X, \mathfrak{s}}^{I} := SW_{X, \mathfrak{s}}\cdot U^{-d/2} = \langle u_{2}^{d/2}, [\mathcal{M}(X, \mathfrak{s})_{\omega}] \rangle \cdot U^{-d/2} \in \mathbb{F}[U^{-1}, Q]/Q^2\]
If $d$ is odd, then
\[SW_{X, \mathfrak{s}}^{I} := \langle u_{2}^{d/2}, [\mathcal{M}(X, \mathfrak{s})_{\gamma}] \rangle \cdot QU^{-\frac{d+1}{2}} \in \mathbb{F}[U^{-1}, Q]/Q^2.\]
\end{definition}

Note that, in the above definition, the variables $U$ and $Q$ are purely formal (albeit suggestive of Theorem~\ref{SWequiv1} below), and the cohomology class $u_{2} \in H^{*}(\mathcal{B}^{*}(X, \mathfrak{s}))$ is, as defined previously, the first Chern class of the principal circle bundle $P \to \mathcal{B}^{*}(X, \mathfrak{s})$ induced by the gauge quotient map along with the homomorphism $\mathcal{G} \to S^1$ given by evaluation at a basepoint $x_0 \in X$. The fact that we are working mod 2 here is not merely to avoid issues of orientation; it is in fact critical to the proof of invariance of $SW_{X, \mathfrak{s}}^{I}$ in the case that $d$ is odd, and the author is not aware of any reason to expect that a $\mathbb{Z}$-valued version would be well-defined (see the proof of Lemma~\ref{ok} below).

For intuitive clarity, we at this point restrict to the simple cases where the virtual dimension of the Seiberg-Witten moduli space
\[d(\mathfrak{s}) := \frac{c_{1}(s)^2 - 2\chi(X) - 3\sigma(X)}{4}\]
is equal to 0 or $-1$. In either case, $SW_{X, \mathfrak{s}}^{I}$ is simply a mod 2 count of solutions. The following proposition tells us that this count is independent of the choices made in the construction.

\begin{lemma}\label{ok} Let $X$ be a closed, smooth four-manifold with spin structure $\mathfrak{s}$. Assume that $b_{2}^{+}(X) \geq 3$, and that $d(\mathfrak{s}) \in \{0, -1\}$. Let $(\omega_{0}, g_{0}), (\omega_{1}, g_{1})$ be two choices of perturbing 2-form and metric, respectively, such that the moduli space of solutions to the associated perturbed Seiberg-Witten equations is regular and contains no reducibles. For $i = 0, 1$, let $\gamma_{i}$ be a path of 2-forms from $\omega_{i}$ to $-\omega_{i}$ such that the moduli space over $\gamma_{i}$ is regular and contains no reducibles.  Then $SW_{X, \mathfrak{s}, \omega_{0}, g_{0}, \gamma_{0}}^{I} = SW_{X, \mathfrak{s}, \omega_{1}, g_{1}, \gamma_{1}}^{I}$.
\end{lemma}
\begin{proof}  If $d(\mathfrak{s}) = 0$, then $SW^{I}_{X, \mathfrak{s}, \omega_{i}, g_{i}, \gamma_{i}} = SW_{X, \mathfrak{s}}$ and invariance follows from the well-known invariance of the Seiberg-Witten invariant. 

We assume henceforth that $d(\mathfrak{s}) = -1$. We may assume without loss of generality that $\gamma_{0}$ and $\gamma_{1}$ are disjoint paths (if not, we may simply choose a third set $(\omega_{2}, g_{2}, \gamma_{2})$ which is disjoint from both). Let $\Gamma(t) = (\omega_{t}, g_{t})$ be a smooth path from $(\omega_{0}, g_{0})$ to $(\omega_{1}, g_{1})$ such that for all $t$, the moduli space of solutions to the perturbed Seiberg-Witten equations on $X$ induced by $\Gamma(t)$ consists of only irreducible solutions. We may choose $\Gamma$ such that, for $i = 0, 1$, the concatenation of $\Gamma$ with $\gamma_{i}$ is smooth. Observe that $-\Gamma(-t)$ provides a path from $-\omega_{1}$ to $-\omega_{0}$.  Concatenating the four paths $\Gamma(t), \gamma_{1}(t), -\Gamma(-t), \gamma_{0}(-t)$, we get a family of perturbation-metric pairs which is smoothly parametrized by $S^1$.  In fact, we also naturally get a family of metrics $g_{x}$ which is parametrized by $D^2$, since $g_{t}$ is constant along $\gamma_{0}$ and $\gamma_{1}$, so by identifying $D^2$ with $[0, 1] \times [0, 1]$ (where the paths each get sent to an edge), we can simply define $g_{(a, b)} = g_{(a, 0)}$ for all $b$. By Lemma 27.1.4 in \cite{Kronheimer_Mrowka_2007}, since $b_{2}^{+}(X) \geq 3$, there exists a family of forms $\omega_{x}$ extending the $S^1$-family above to all of $D^2$, such that the moduli space of solutions $\mathcal{M}(X, \mathfrak{s})_{D^2}$ is again regular and contains no reducibles.

Now, the moduli space $\mathcal{M}(X, \mathfrak{s})_{D^2}$ has dimension $\text{dim}(D^2) + d(\mathfrak{s}) = 2 - 1 = 1$, and every one-manifold has an even number of points in its boundary. But solutions along $\Gamma(t)$ are paired with solutions along $-\Gamma(-t)$, and there are no solutions at the corner points by assumption of regularity (and the fact that $d(\mathfrak{s}) = -1$), so it follows that:
\[\#\mathcal{M}(X, \mathfrak{s})_{\gamma_{0}} + \#\mathcal{M}(X, \mathfrak{s})_{\gamma_{1}} \equiv 0 \text{ mod 2.}\]
The statement is proven.  \end{proof}

Proving the general case, where $d(\mathfrak{s})$ is not necessarily $0$ or $-1$, involves using a \textit{cylinder} (rather than a rectangle) of perturbations, produced by first adding on $-\gamma_{0}$ and $-\gamma_{1}$ to produce complete loops on either end, and \textit{then} invoking 27.1.4 from \cite{Kronheimer_Mrowka_2007}. One can then consider the $\mathbb{Z}_{2}$-equivariant cohomology pairing with $u_{2}$, and invoke an appropriate (equivariant) version of Stokes' theorem.  We omit the details.

The primary goal of this subsection is to show:
\begin{theorem}\label{SWequiv1} If $b_{2}^{+}(X) \geq 4$, then $\mathfrak{m}^{I}_{X, \mathfrak{s}}(1) = SW_{X, \mathfrak{s}}^{I}$, where we've made the identification $\widecheck{\HMI}(S^3) \cong \mathbb{F}[U^{-1}, Q]/Q^2$. In particular, $\mathfrak{m}^{I}_{X, \mathfrak{s}}$ is an invariant of the pair $(X, \mathfrak{s})$.
\end{theorem}

\subsubsection{Defining $\overrightarrow{\HMI}(X, \mathfrak{s})$} In order to relate $\mathfrak{m}_{X, \mathfrak{s}}^{I}$ to $SW_{X, \mathfrak{s}}^{I}$, we will go through the involutive analogue of Kronheimer and Mrowka's map 
\[\overrightarrow{\HM}(X, \mathfrak{s}): \widehat{\HM}(S^3) \to \widecheck{\HM}(S^3),\]
which counts only irreducible solutions asymptotic to reducibles on either end. More precisely, on the chain level, the map $\overrightarrow{\HM}(X, \mathfrak{s})$ is simply the map $m_{s}^{u}: \mathfrak{C}^{u} \to \mathfrak{C}^{s}$ which counts irreducible solutions in $X$ asymptotic to unstable reducible critical points on the incoming end and stable reducible critical points on the outgoing end. We define the involutive analogue as follows: let $\omega$ be a choice of perturbing 2-form on $X \setminus (B^4 \sqcup B^4)$ such that the moduli spaces $\mathcal{M}([\mathfrak{a}], X, \mathfrak{s}, [\mathfrak{b}])_{\omega}$ asymptotic to $[\mathfrak{a}] \in \widehat{\HM}(S^3)$ and $[\mathfrak{b}] \in \widecheck{\HM}(S^3)$ are regular and contain no reducibles. This induces a map $\overrightarrow{\CM}(X, \mathfrak{s}, \omega)$. Choosing $-\omega$ as the perturbing form instead gives the map $\overrightarrow{\CM}(X, \mathfrak{s}, -\omega)$. Choosing a path $\gamma$ from $\omega$ to $-\omega$ such that the moduli space over $\gamma$ is regular and contains no reducibles induces a chain homotopy $\overrightarrow{h}(X, \mathfrak{s}, \gamma)$ from $\overrightarrow{\CM}(X, \mathfrak{s}, \omega)$ to $\overrightarrow{\CM}(X, \mathfrak{s}, -\omega)$. We define
\[\overrightarrow{\CMI}(X, \mathfrak{s}, \gamma) = \begin{pmatrix} \overrightarrow{\CM}(X, \mathfrak{s}, \omega) & 0 \\ \overrightarrow{h}(X, \mathfrak{s}, \gamma) & \overrightarrow{\CM}(X, \mathfrak{s}, -\omega) \end{pmatrix}\]
where elements of the involutive chain complex $a + Qb$ are identified with vectors of the form $(a, b)$. The induced map on homology is independent of $\gamma$.

\subsubsection{Proving Theorem~\ref{SWequiv1}} There will be two steps to proving Theorem~\ref{SWequiv1}: first, we will relate $\overrightarrow{\HMI}(X, \mathfrak{s}, \gamma)$ to $SW_{X, \mathfrak{s}}^{I}$ by showing that evaluating the former at $1 \in \widehat{\HMI}(S^3)$ gives the latter (this is a standard ``neck-stretching" argument). Then, we will relate $\mathfrak{m}_{X, \mathfrak{s}}^{I}$ to $\overrightarrow{\HMI}(X, \mathfrak{s}, \gamma)$ via a composition formula for the latter (through involutively admissible cuts).

\begin{theorem} Let $X$ be a closed, smooth four-manifold with spin structure $\mathfrak{s}$. Assume that $b_{2}^{+}(X) \geq 3$. Then $\overrightarrow{\HMI}(X, \mathfrak{s})(1) = SW_{X, \mathfrak{s}}^{I}$.
    
\end{theorem}

\begin{proof} The proof is entirely analogous to that of Proposition 27.4.1 in \cite{Kronheimer_Mrowka_2007}, so we omit the details. The idea is to ``stretch the neck" near the spherical ends of $X \setminus (B^4 \sqcup B^4)$. More precisely, the goal is to construct a compactified moduli space which fibers over the end-stretching parameter $S \in [0, \infty]$, show that the codimension 1 strata of the moduli space are precisely the moduli spaces $M(X, \mathfrak{s})^{I}$ (the fiber over $S = 0$) and $M([a_{-1}], X^{*}, \mathfrak{s}^{*}, [a_{0}])^{I}$ (contained in the fiber over $S = \infty$), and then apply Stokes' theorem (as it appears in the book) to say that these two strata must have the same count (mod 2). 

To clarify some notation: $a_{0}$ is the critical point in the blown-up configuration space $\mathcal{B}_{k}^{\sigma}(S^3, s_{0})$ corresponding to the first positive eigenvalue of the perturbed Dirac operator $D_{B, \mathfrak{q}}$, which in turn arises from the unique critical point $[B, 0]$ of the perturbed Chern-Simons-Dirac functional $\mathcal{L} + \mathfrak{q}$, where $\mathfrak{q}$ is a sufficiently small admissible perturbation.\end{proof}
Now, to finish off the proof of Theorem~\ref{SWequiv1}, we turn to proving a composition formula.
\begin{proposition} Let $X$ be a smooth, closed, oriented four-manifold with $b_{2}^{+} \geq 4$ and equipped with a self-conjugate $\text{spin}^{c}$ structure $\mathfrak{s}$, and let $N \subset X$ be an involutively admissible cut splitting $X$ into $X_{0}$ and $X_{1}$. Then
\[\overrightarrow{\HMI}(X, \mathfrak{s}) = \widecheck{\HMI}(X_{1}, \mathfrak{s} |_{X_{1}}) \circ \overrightarrow{\HMI}(X_{0}, \mathfrak{s}|_{X_{0}}).\]
\end{proposition}
\begin{proof} Let $\omega$ be a perturbing 2-form on $X$, and let $\gamma$ denote a choice of path from $\omega$ to $-\omega$. We will denote by $\mathfrak{s}_{i}$ the $\text{spin}^{c}$ structure $\mathfrak{s}$ restricted to $X_{i}$, by $\omega_{i}$ the perturbing 2-form $\omega$ restricted to $X_{i}$, and by $\gamma_{i}$ the path $\gamma$ restricted (at each time $t$) to $X_{i}$. The right side of the equation we're trying to establish can be written as
\[\begin{pmatrix} \widecheck{\CM}(X_1, \mathfrak{s}_{1}, \omega_{1}) & 0 \\ \widecheck{h}(X_1, \mathfrak{s}_{1}, \gamma_1) & \widecheck{\CM}(X_{1}, \mathfrak{s}_{1}, -\omega_1) \end{pmatrix} \cdot \begin{pmatrix} \overrightarrow{\CM}(X_0, \mathfrak{s}_{0}, \omega_{0}) & 0 \\ \overrightarrow{h}(X_0, \mathfrak{s}_{0}, \gamma_0) & \overrightarrow{\CM}(X_{0}, \mathfrak{s}_{0}, -\omega_0)\end{pmatrix}\]
Multiplying this out, we get $\widecheck{\CM}(X_1, \mathfrak{s}_{1}, \omega_{1}) \circ \overrightarrow{\CM}(X_0, \mathfrak{s}_{0}, \omega_{0})$ and $\widecheck{\CM}(X_{1}, \mathfrak{s}_{1}, -\omega_1) \circ \overrightarrow{\CM}(X_{0}, \mathfrak{s}_{0}, -\omega_0)$ along the diagonal, which are equal to $\overrightarrow{\CM}(X, \mathfrak{s}, \omega)$ and $\overrightarrow{\CM}(X, \mathfrak{s}, -\omega)$, respectively, by the non-involutive composition formula (Theorem 3.5.3 in \cite{Kronheimer_Mrowka_2007}). The upper-right entry clearly vanishes, and the lower-left entry is given by:
\[\widecheck{h}(X_1, \mathfrak{s}_{1}, \gamma_1) \circ \overrightarrow{\CM}(X_0, \mathfrak{s}_{0}, \omega_{0}) + \widecheck{\CM}(X_{1}, \mathfrak{s}_{1}, -\omega_1) \circ \overrightarrow{h}(X_0, \mathfrak{s}_{0}, \gamma_0).\]
Now, following the proof of the composition law in \cite{Kronheimer_Mrowka_2007}, we ``stretch the neck" along $N \times \mathbb{R}$ (identified with the normal bundle of $N$), and choose a perturbing 2-form $\widetilde{\omega}$ which is supported in the interior of the two halves $X_0$, $X_{1}$. Let $\widetilde{\gamma}$ be a path of perturbations given by concatenating two paths $\widetilde{\Gamma}_{0}$ and $\widetilde{\Gamma}_{1}$.  We define $\widetilde{\Gamma}_{0}$ so that, when restricted to $X_{1}$, it is a constant path with value $-\omega_1$ and, when restricted to $X_{0}$, it provides a path from $-\omega_{0}$ to $\omega_{0}$. Meanwhile, we define $\widetilde{\Gamma}_{1}$ so that, when restricted to $X_{0}$, it is a constant path with value $\omega_{0}$ and, when restricted to $X_{1}$, it provides a path from $-\omega_{1}$ to $\omega_1$. Now, we observe that the map on Floer homology induced by the path $\widetilde{\Gamma}_{0}$ is exactly $\widecheck{\CM}(X_{1}, \mathfrak{s}_{1}, -\widetilde{\omega}_1) \circ \overrightarrow{h}(X_0, \mathfrak{s}_{0}, \widetilde{\gamma}_0)$, while the map on Floer homology induced by the path $\widetilde{\Gamma}_{1}$ is exactly $\widecheck{h}(X_1, \mathfrak{s}_{1}, \widetilde{\gamma}_1) \circ \overrightarrow{\CM}(X_0, \mathfrak{s}_{0}, \widetilde{\omega}_{0})$. The map induced by $\widetilde{\gamma}$ is equal to the sum of the two, as desired.\end{proof}
\subsection{Relation to $SW^{Pin(2)}_{X, \mathfrak{s}}$}
Involutive monopole Floer theory, and more specifically the invariant $SW_{X, \mathfrak{s}}^{I}$, captures just some of the information afforded by the charge conjugation symmetry $\jmath$ which exists when $\mathfrak{s}$ is self-conjugate. Indeed, this is part of a larger symmetry with respect to the action of $Pin(2) = S^1 \cup jS^1 \subset \mathbb{H}$, where $j$ acts as $\jmath$. In \cite{baraglia2023mod}, Baraglia constructs a $Pin(2)$-equivariant Seiberg-Witten invariant $SW^{Pin(2)}_{X, \mathfrak{s}}$, which takes the form of a $Q$-equivariant map:
\[SW^{Pin(2)}_{X, \mathfrak{s}}: H^{*}_{Pin(2)}(pt) \cong \mathbb{F}[V, Q]/Q^3 \to H^{*}(\mathbb{T}(X); \mathbb{F})[Q]/(Q^{b_{2}^{+}(X)})\]
where $\mathbb{T}(X) = H^1(X; \mathbb{R})/H^1(X;\mathbb{Z})$ is the Picard torus associated to $X$. For simplicity, we shall assume that $X$ is simply-connected, so that this is a map:
\[SW^{Pin(2)}_{X, \mathfrak{s}}: \mathbb{F}[V, Q]/Q^3 \to \mathbb{F}[Q]/Q^{b_{2}^{+}(X)}.\]
The following theorem establishes the relation between $SW_{X, \mathfrak{s}}^{I}$ and $SW^{Pin(2)}_{X, \mathfrak{s}}$.
\begin{theorem}\label{baraginvol} Let $(X, \mathfrak{s})$ be a closed, simply-connected, spin four-manifold with $b_{2}^{+}(X) \geq 3$, and suppose that $d(\mathfrak{s}) \in \{0, -1\}$. Then $SW_{X, \mathfrak{s}}^{I} = [SW_{X, \mathfrak{s}}^{Pin(2)}(1)] \in \mathbb{F}[U, Q]/Q^2$, where $[SW_{X, \mathfrak{s}}^{Pin(2)}(1)]$ denotes $SW_{X, \mathfrak{s}}^{Pin(2)}(1)$ mod $Q^2$ (i.e. truncated after its first order $Q$ term).
\end{theorem}
\begin{proof} Firstly, if $d(\mathfrak{s}) = 0$, then this is clear, since $SW_{X, \mathfrak{s}}^{I}$ is just the ordinary Seiberg-Witten invariant in this case, and by degree considerations $[SW_{X, \mathfrak{s}}^{Pin(2)}(1)]$ consists of just the $Q^0$ term, which is the ordinary Seiberg-Witten invariant by virtue of the fact that the Bauer-Furuta stable cohomotopy invariant recovers the Seiberg-Witten invariant in cohomology.

Suppose henceforth that $d(\mathfrak{s}) = -1$. We will follow Baraglia's notation throughout this proof; specifically, $f: S^{V, U} \to S^{V', U'}$ denotes the $Pin(2)$-equivariant monopole map, where $S^{V, U}$ and $S^{V', U'}$ are fiberwise one-point compactifications of the vector bundles $V \oplus U$ and $V' \oplus U'$, respectively (described in section 2.1.5 and 2.1.6 of the introduction), over the base space $B = S(H^{+}(X))$ (the unit sphere inside the maximal $Q_{X}$-positive-definite subspace of $H^2(X)$). Furthermore, we denote by $Y^{V, U}$ the compact manifold with boundary given by $(S^{V, U} \setminus \text{nbhd}(S^{U}))/S^1$. 

By definition, $SW_{X, \mathfrak{s}}^{Pin(2)}(1)$ is equal to $(\pi_{V, U})_{*}(f^{*}(\tau_{V', U'}^{\phi_{\text{taut}}}))$, where 
\[(\pi_{V, U})_{*}: H^{*}_{\mathbb{Z}_{2}}(Y^{V, U}, \partial Y^{V, U}) \to H^{*}_{\mathbb{Z}_{2}}(B)\]
is the pushforward map induced by the projection and $\tau_{V', U'}^{\phi_{\text{taut}}} \in H^{*}_{Pin(2)}(S^{V', U'}, S^{U'})$ is the (lifted) Thom class of the normal bundle of the tautological section $\phi_{\text{taut}}: S(H^{+}(X)) \to S(H^{+}(X)) \times H^{+}(X)$ given by $\alpha \mapsto (\alpha, \alpha)$. Note that, since $X$ is simply-connected, the $\mathbb{Z}_2$ action on both the base space $B$ and on $Y^{V, U}$ is free. The quotient of $B$ with respect to the action is $\mathbb{RP}^{b_{2}^{+}(X) - 1}$. We will denote the quotient of $Y^{V, U}$ by $\overline{Y^{V, U}}$. Thus, we may consider $(\pi_{V, U})_{*}$ to be a map from $H^{*}(\overline{Y^{V, U}}, \partial \overline{Y^{V, U}})$ to $H^{*}(\mathbb{RP}^{b_{2}^{+}(X) - 1})$.

Since the Thom class of the normal bundle of a submanifold is dual to the homology class of the submanifold itself, we have that $f^{*}(\tau_{V', U'}^{\phi_{\text{taut}}})$ is dual in $(\overline{Y^{V, U}}, \partial \overline{Y^{V, U}})$ to $f^{-1}(\phi_{\text{taut}}(B))$ (possibly after perturbing to achieve transversality). Meanwhile, we can consider the pullback $(\pi_{V, U})^{*}(Q^{b_{2}^{+}(X)-2})$, where $Q \in \mathbb{F}[Q]/(Q^{b_{2}^{+}(X)}) = H^{*}(\mathbb{RP}^{b_{2}^{+}(X) - 1})$. The class $(\pi_{V, U})^{*}(Q^{b_{2}^{+}(X)-2})$ is dual to $\pi_{V, U}^{-1}(PD(Q^{b_{2}^{+}(X)-2})) = \pi_{V, U}^{-1}(\mathbb{RP}^{1})$, where $\mathbb{RP}^{1}$ generates the first homology of $\mathbb{RP}^{b_{2}^{+}(X) - 1}$. Note that $f^{-1}(\phi_{\text{taut}}(B))$ and $\pi_{V, U}^{-1}(\mathbb{RP}^{1})$ are complementary-dimensional by virtue of the fact that $d(\mathfrak{s}) = -1$. We have the following equality (mod 2):
\[\#\left(f^{-1}(\phi_{\text{taut}}(B)) \cap \pi_{V, U}^{-1}(\mathbb{RP}^{1})\right) = \langle f^{*}(\tau_{V', U'}^{\phi_{\text{taut}}}) \smile (\pi_{V, U})^{*}(Q^{b_{2}^{+}(X)-2}), [\overline{Y^{V, U}}, \partial \overline{Y^{V, U}}] \rangle.\]
Fix a choice of perturbation $\omega \in \Omega^{+}(X)$. We may choose $\omega$ so that it lies in $S(\mathcal{H}^{+}(X))$, the unit sphere in the space of positive harmonic 2-forms on $X$. Let $\gamma \subset S(\mathcal{H}^{+}(X))$ be a choice of path from $\omega$ to $-\omega$. We can think of $\mathbb{RP}^1 \subset \mathbb{RP}^{b_{2}^{+}(X) - 1}$ as being the image under the $\mathbb{Z}_{2}$-quotient of the path $\gamma$. A point $x \in \pi_{V, U}^{-1}(\mathbb{RP}^{1})$ is in $f^{-1}(\phi_{\text{taut}}(B))$ if and only if it is a solution to the Seiberg-Witten equations perturbed by $\pi_{V, U}(x) \in S(H^{+}(X))$. Thus, with generically chosen $\omega$ (so that there is no solution over $\omega$), the count $\#\left(f^{-1}(\phi_{\text{taut}}(B)) \cap \pi_{V, U}^{-1}(\mathbb{RP}^{1})\right)$ is precisely equal to $SW_{X, \mathfrak{s}}^{I}$. It follows that the cup product $f^{*}(\tau_{V', U'}^{\phi_{\text{taut}}}) \smile (\pi_{V, U})^{*}(Q^{b_{2}^{+}(X)-2})$ is equal to the unique nontrivial element of the top cohomology of $(\overline{Y^{V, U}}, \partial \overline{Y^{V, U}})$ with $\mathbb{Z}_2$ coefficients if and only if $SW_{X, \mathfrak{s}}^{I} = Q$, and it is zero if and only if $SW_{X, \mathfrak{s}}^{I} = 0$. 

On the other hand, we have that:
\[SW_{X,\mathfrak{s}}^{Pin(2)}(1) = (\pi_{V, U})_{*}(f^{*}(\tau_{V', U'}^{\phi_{\text{taut}}}))\]
so that, by the push-pull formula:
\[(\pi_{V, U})_{*}\left(f^{*}(\tau_{V', U'}^{\phi_{\text{taut}}}) \smile (\pi_{V, U})^{*}(Q^{b_{2}^{+}(X)-2})\right) = SW_{X,\mathfrak{s}}^{Pin(2)}(1) \smile Q^{b_{2}^{+}(X)-2}.\]
The pushforward maps the top cohomology of $(\overline{Y^{V, U}}, \partial \overline{Y^{V, U}})$ isomorphically onto the top cohomology of $\mathbb{RP}^{b_{2}^{+}(X) - 1}$, so we conclude that: $SW_{X,\mathfrak{s}}^{Pin(2)}(1) = Q$ if and only if
\[(\pi_{V, U})_{*}\left(f^{*}(\tau_{V', U'}^{\phi_{\text{taut}}}) \smile (\pi_{V, U})^{*}(Q^{b_{2}^{+}(X)-2})\right) = Q^{b_{2}^{+}(X)-1}\]
which, by our argument above, is in turn true if and only if $SW_{X, \mathfrak{s}}^{I} = Q$. Since both $SW_{X, \mathfrak{s}}^{I}$ and $SW_{X, \mathfrak{s}}^{Pin(2)}$ can only be either $Q$ or $0$ in this case, this completes the proof of the theorem. \end{proof}

\section{Stable Adjunction for General Surfaces}
\begin{proposition}\label{vanishred} Let $\mathfrak{s}$ be a self-conjugate $\text{spin}^c$ structure on a three-manifold $Y$. Suppose that $\HM_{\text{red}}(Y, \mathfrak{s}) = 0$ and $\overline{\iota}_{*} = 1_{\overline{\HM}(Y, \mathfrak{s})}$. Then $\HMI_{\text{red}}(Y, \mathfrak{s}) = 0$.
\end{proposition}
\begin{proof} Consider the commutative diagram:
\[\begin{tikzcd}
	& {Q\widehat{\textit{HM}}(Y, \mathfrak{s})[-1]} & {\widehat{\textit{HMI}}(Y, \mathfrak{s})} & {\widehat{\textit{HM}}(Y, \mathfrak{s})} \\
	{\overline{\textit{HM}}(Y, \mathfrak{s})} & {Q\overline{\textit{HM}}(Y, \mathfrak{s})[-1]} & {\overline{\textit{HMI}}(Y, \mathfrak{s})} & {\overline{\textit{HM}}(Y, \mathfrak{s})}
	\arrow["{p_{*}^{I}}", from=1-3, to=2-3]
	\arrow["{p_{*}}", from=1-2, to=2-2]
	\arrow["{p_{*}}", from=1-4, to=2-4]
	\arrow["{\widehat{g}_{*}}", from=1-2, to=1-3]
	\arrow["{\widehat{h}_{*}}", from=1-3, to=1-4]
	\arrow["{\overline{g}_{*}}", from=2-2, to=2-3]
	\arrow["{\overline{h}_{*}}", from=2-3, to=2-4]
	\arrow["{Q(1+\overline{\iota}_{*})}", from=2-1, to=2-2]
\end{tikzcd}\]
The monopole analogue of Lemma~\ref{descend} (for which the same proof goes through) implies that $\widehat{h}_{*}(\widehat{\HMI}_{\text{red}}(Y, \mathfrak{s})) \subset \widehat{\HM}_{\text{red}}(Y, \mathfrak{s}) = 0$, which by exactness of the top row implies that $\widehat{\HMI}_{\text{red}}(Y, \mathfrak{s}) \subset \text{im }\widehat{g}_{*}$. Meanwhile, focusing on the leftmost square, notice that $i_{*}$ is injective since $\ker p_{*} = \widehat{\HM}_{\text{red}} = 0$ and $\overline{g}_{*}$ is injective since $\ker \overline{g}_{*} = \text{im } Q(1 + \overline{\iota}_{*}) = 0$, hence their composition is injective. So the map obtained by going around the square in the other direction ought to also be injective. In particular, $p^{I}_{*}$ is injective on the image of $\widehat{g}_{*}$. In other words, $\text{im }\widehat{g}_{*} \cap \ker p^{I}_{*} = 0$. But $\ker p^{I}_{*} = \widehat{\HMI}_{\text{red}}(Y, \mathfrak{s})$ by definition, so 
\[0 = \text{im }\widehat{g}_{*} \cap \widehat{\HMI}_{\text{red}}(Y, \mathfrak{s}) = \widehat{\HMI}_{\text{red}}(Y, \mathfrak{s}),\]
by virtue of the fact that $\widehat{\HMI}_{\text{red}}(Y, \mathfrak{s}) \subset \text{im }\widehat{g}_{*}$. \end{proof}
Note that the Heegaard Floer analogue of Proposition~\ref{vanishred} holds as well, with an identical proof. We stated it using monopole Floer theory because the author is only aware of how to show that $\overline{\iota}_{*}$ is the identity under sufficiently general conditions on the monopole side. 

The following proposition provides the final substantial technical ingredient for proving Theorem~\ref{SWstabadjintro}.
\begin{proposition}\label{tech} Let $Y$ be a closed, oriented three-manifold with vanishing triple cup product (i.e. the map $\Lambda^{*}H^1(Y; \mathbb{Z}) \to \mathbb{Z}$ given by $(\alpha_1, \alpha_2, \alpha_3) \mapsto \langle \alpha_1 \cup \alpha_2 \cup \alpha_3, [Y] \rangle$ is zero), and let $\mathfrak{s}$ be a self-conjugate $\text{spin}^{c}$ structure on $Y$. Then the involution $\overline{\iota}_{*}: \overline{\HM}(Y, \mathfrak{s}) \to \overline{\HM}(Y, \mathfrak{s})$ is the identity.
\end{proposition}
\begin{proof} Let $\mathbb{T}$ be the torus of reducible solutions to the Seiberg-Witten equations on $Y$ up to gauge, identified with $\mathbb{T}(Y) = H^1(Y; \mathbb{R})/H^1(Y; \mathbb{Z})$ by fixing a base solution $[(A_0, 0)]$. Let $f: \mathbb{T} \to \mathbb{R}$ be a Morse function on $\mathbb{T}$, and let $\mathfrak{q}$ be a tame perturbation with gradient zero on the reducible locus such that at every critical point $[(A, 0)] \in \mathbb{T}$ of $f$, the corresponding Dirac operator $D_{A, \mathfrak{q}}$ has simple spectrum. Let $\mathfrak{p} = \text{grad}(f \circ p) + \mathfrak{q}$, where $p: \mathcal{B}(Y, \mathfrak{s}) \to \mathbb{T}$ is the retraction $[(A_0 + a \otimes 1_{S}, \phi)] \mapsto [a_{\text{harm}}] \in \mathbb{T}(Y)$. The chain complex $\overline{\CM}(Y, \mathfrak{s}, \mathfrak{p})$ is precisely the coupled Morse chain complex $\overline{C}_{*}(\mathbb{T}, D_{-, \mathfrak{q}}, f; \mathbb{F})$ associated to the Morse function $f: \mathbb{T} \to \mathbb{R}$ and the family of Dirac operators parametrized by $\mathbb{T}$ via $[(A, 0)] \mapsto D_{A, \mathfrak{q}}$.  As an $\mathbb{F}[U]$-module, we have:
\[\overline{C}_{*}(\mathbb{T}, D_{-, \mathfrak{q}}, f; \mathbb{F}) \cong C_{*}^{\text{Morse}}(\mathbb{T}, f) \otimes \mathbb{F}[U, U^{-1}]\]
where $U$, as usual, has degree $-2$. In this isomorphism, we have identified the generator $([(A, 0)], [\phi_{i}])$ on the left side, where $[(A, 0)] \in \text{Crit}(f)$ and $\phi_{i}$ is the eigenvector associated to the $i$th eigenvalue $\lambda_{i}$ of $D_{A, \mathfrak{q}}$, with $[(A, 0)] \otimes U^i$ (which we will write as $[(A, 0)]U^{-i}$).

The charge conjugation map $\jmath|_{\mathbb{T}}: \mathbb{T} \to \mathbb{T}$ acts as the inversion map with respect to the group structure on $\mathbb{T} = H^1(Y; \mathbb{R})/H^1(Y; \mathbb{Z})$. In other words, the map is reflection on each $S^1$ factor. It induces a canonical isomorphism of chain complexes $\jmath_{*}: \overline{\CM}(Y, \mathfrak{s}, \mathfrak{p}) \to \overline{\CM}(Y, \mathfrak{s}, \jmath_{*}\mathfrak{p})$ (by sending generators of the former to generators of the latter, as a map between configuration spaces), which in the framework of the coupled Morse complex is an isomorphism:
\[\jmath_{*}: \overline{C}_{*}(\mathbb{T}, D_{-, \mathfrak{q}}, f) \overset{\cong}{\la} \overline{C}_{*}(\mathbb{T}, D_{-, \jmath_{*}\mathfrak{q}}, f \circ \jmath^{-1}).\]
To be precise, $\jmath_{*}$ identifies the generator $([(A, 0)], [\phi_{i}])$ on the left side with the generator $(\jmath([(A, 0)]), [\psi_{i}])$ on the right side, where $\psi_{i}$ is the eigenvector associated to the $i$th eigenvalue $\mu_{i}$ of $D_{-, \jmath_{*}\mathfrak{q}}$.

A cylinder $Y \times \mathbb{R}_t$, with perturbation $\jmath_{*}\mathfrak{p}$ whenever $t \leq 0$ (the incoming end) and perturbation $\mathfrak{p}$ whenever $t \geq 1$ (the outgoing end), along with interior perturbation interpolating between the two, induces a chain homotopy equivalence on the Floer chain complexes $\overline{\CM}(Y, \mathfrak{s}, \jmath_{*}\mathfrak{p}) \to \overline{\CM}(Y, \mathfrak{s}, \mathfrak{p})$. We can choose the interpolating family of perturbations to be of the form:
\[\mathfrak{p}_{t} = \text{grad}(f_{t} \circ p) + \mathfrak{q}_{t}\]
where $f_{t} = f \circ \jmath^{-1}$ whenever $t \leq -2$, $f_{t} = f$ whenever $t \geq -1$, $\mathfrak{q}_{t} = \mathfrak{q} \circ \jmath$ whenever $t \leq 0$, and $\mathfrak{q}_{t} = \mathfrak{q}$ whenever $t \geq 1$ (so we first modify the Morse function on $\mathbb{T}$, and then modify the perturbation orthogonal to $\mathbb{T}$). In this case, the interpolating family gives rise to a coupled Morse continuation map (a chain homotopy equivalence):
\[\Phi_{L(t, -), f_t}: \overline{C}_{*}(\mathbb{T}, D_{-, \jmath_{*}\mathfrak{q}}, f \circ \jmath^{-1}; \mathbb{F}) \to \overline{C}_{*}(\mathbb{T}, D_{-, \mathfrak{q}}, f; \mathbb{F}).\]
Here $L(t, -)$ denotes the homotopy of families of self-adjoint Fredholm operators over $\mathbb{T}$ induced by $\mathfrak{q}_{t}$, i.e. $L(t, -) = D_{-, \mathfrak{q}_t}$. The map $\Phi_{L(t, -), f_t}$ is exactly as described in Proposition 33.3.8 of \cite{Kronheimer_Mrowka_2007}, except now we are changing the Morse function as well, so it is given by a composition of: (1) a standard Morse continuation map (inducing the canonical isomorphism on Morse homology given by changing the Morse function); and (2) a coupled Morse continuation map given by counting (projective classes of) solutions $(\gamma, [\phi])$ to the system:
\begin{equation}\label{morse} \frac{d}{dt}\gamma + (\text{grad } f)_{\gamma(t)} = 0 \end{equation} \begin{equation} \gamma^{*}(\nabla)\phi + (L(t, \gamma(t))\phi)dt = 0. \end{equation}
Here, $\nabla$ is a fixed connection on the principal $U(L^{2}(S):H^1(S))$-bundle associated to the pair of Hilbert bundles $(\mathcal{H}, \mathcal{H}_{1}) \cong (\mathbb{T} \times L^2(S), \mathbb{T} \times H^1(S))$ over $\mathbb{T}$.

Our objective is to show that $\overline{\iota} := \Phi_{L(t, -), f_{t}} \circ \jmath_{*}: \overline{C}_{*}(\mathbb{T}, D_{-, \mathfrak{q}}, f) \to \overline{C}_{*}(\mathbb{T}, D_{-, \mathfrak{q}}, f)$ induces the identity on homology. We can decompose $\Phi_{L(t,-), f_t}$ as follows:
\[\Phi_{L(t,-), f_t} = \Phi_0 + U^{-1}\Phi_{2} + U^{-2}\Phi_{4} + \dots\]
where $\Phi_{i}: C_{*}^{\text{Morse}}(\mathbb{T}, f\circ \jmath^{-1}) \to C_{*}^{\text{Morse}}(\mathbb{T}, f)$ decreases degree by $i$. The fact that there are no terms of the form $U^n \Phi_{-2n}$ for $n > 0$ is due to the fact that the coupled Morse continuation counts solutions $(\gamma, [\phi])$ where $\gamma$ must satisfy equation~\ref{morse}, which is the \textit{downward} gradient flow of $f$. We remark that $\Phi_0$ is just the standard Morse continuation map associated to $f_{t}$. Moreover, the induced map $\overline{\iota}_{*}$ on homology is independent of the choice of path $f_{t}$ and dependent only on the homotopy class of the path $L(t, -) = D_{-, \mathfrak{q}_{t}}$; the same is true for the induced maps $(\Phi_{i} \circ \jmath_{*})_{*}$ on homology for all $i$. We first show that $\Phi_{0} \circ \jmath_{*}$ is the identity on homology.

\begin{lemma}\label{iden} Given any closed, oriented three-manifold $Y$, the map $\Phi_0 \circ \jmath_{*}$ induces the identity on homology. In fact, we can choose $f$ and $f_{t}$ so that $\Phi_0 \circ \jmath_{*}$ is the identity on the chain level.
\end{lemma}
\begin{proof} We define the function $h: \mathbb{R} \to \mathbb{R}$ by $h(x) = \cos(2\pi x)$, so that it looks like the ``standard" Morse function when descended to $S^1 = \mathbb{R}/\mathbb{Z}$ (with unique max at $[0]$ and unique min at $[1/2]$).  Let the Morse function $f: \mathbb{T} \to \mathbb{R}$ be given by
\[f(r_1, \dots, r_{b_1}) = h(r_1) + \dots + h(r_{b_1}).\]
Observe that $\jmath$ is the map $(r_1, \dots, r_{b_1}) \mapsto (-r_1, \dots, -r_{b_{1}})$, so $f = f \circ \jmath^{-1}$. Thus, $\Phi_{0}$ is the Morse continuation map associated to the pair $(f, f)$, which is the identity if we just choose a slight perturbation of the constant homotopy. If we consider $\Phi_{0}$ to be a map $\overline{C}_{*}(\mathbb{T}, D_{-, \jmath_{*}\mathfrak{q}}, f) \to \overline{C}_{*}(\mathbb{T}, D_{-, \mathfrak{q}}, f)$, this means that it identifies the generator $(p, [\psi_{i}])$ on the left side, where $p$ is a critical point of $f$ and $\psi_{i}$ is the $i$th eigenvector of $D_{-, \jmath_{*}\mathfrak{q}}$, with the generator $(p, [\phi_{i}])$ on the right side, where $\phi_{i}$ is the $i$th eigenvector of $D_{-, \mathfrak{q}}$.

As we described above, the map $\jmath_{*}: \overline{C}_{*}(\mathbb{T}, D_{-, \mathfrak{q}}, f) \overset{\cong}{\la} \overline{C}_{*}(\mathbb{T}, D_{-, \jmath_{*}\mathfrak{q}}, f \circ \jmath^{-1})$ is exactly the inverse to this. So, with the above choice of $f$, the map $\Phi_{0} \circ \jmath_{*}: \Phi_{L(t, -), f_{t}} \circ \jmath_{*}: \overline{C}_{*}(\mathbb{T}, D_{-, \mathfrak{q}}, f) \to \overline{C}_{*}(\mathbb{T}, D_{-, \mathfrak{q}}, f)$ is the identity. \end{proof}

Our goal is now to show that $\Phi_{i} \circ \jmath_{*}$ is the zero map on homology for all $i > 0$. In order to do this, we will briefly enter the world of $\mathbb{Z}$ coefficients. We will make use of the following fact: the map $\overline{\iota}^{\mathbb{Z}} = \Phi_{D_{-, \mathfrak{q}_{t}}, f_{t}}^{\mathbb{Z}} \circ \jmath_{*}$ is a homotopy involution on $\overline{\CM}(Y, \mathfrak{s}, \mathfrak{p}; \mathbb{Z})$, the bar-flavor chain complex with \textit{coefficients in $\mathbb{Z}$}. The proof of this is the same as for the analogous result in the case of $\mathbb{Z}/2\mathbb{Z}$ coefficients (see e.g. \cite{MR3649355} for a proof in the Heegaard Floer case, which is formally identical). This fact relies on the existence of a \textit{transitive system} of continuation maps over $\mathbb{Z}$, which follows from the naturality of the monopole Floer homology groups.

Writing the $\mathbb{Z}$-coefficient continuation map $\Phi_{D_{-, \mathfrak{q}_{t}}, f_{t}}^{\mathbb{Z}}$ as:
\[\Phi_{D_{-, \mathfrak{q}_{t}}, f_{t}}^{\mathbb{Z}} = \Phi_{0}^{\mathbb{Z}} + U^{-1}\Phi_{2}^{\mathbb{Z}} + U^{-2}\Phi_{4}^{\mathbb{Z}} + \dots\]
we define homogeneous maps $\overline{\iota}_{i}^{\mathbb{Z}} := \Phi_{i}^{\mathbb{Z}} \circ \jmath_{*}$, so:
\[\overline{\iota}^{\mathbb{Z}} = \overline{\iota}_{0}^{\mathbb{Z}} + U^{-1}\overline{\iota}_{2}^{\mathbb{Z}} + U^{-2}\overline{\iota}_{4}^{\mathbb{Z}} + \dots\]

Fix $f$ and $f_{t}$ as in the proof of Lemma~\ref{iden}. It follows from Lemma~\ref{iden} that $\overline{\iota}_{0}^{\mathbb{Z}} = \text{Id}$ (the proof goes through for $\mathbb{Z}$ coefficients). 

We have the following result from \cite{Kronheimer_Mrowka_2007} (as a special case of Theorem 35.1.1; see the remark on manifolds with vanishing triple cup product in section 35.3 for enhancing the result to $\mathbb{Z}$ coefficients in the special case):
\begin{lemma}\label{difftriv} Assume $Y$ is a closed, oriented three-manifold with vanishing triple cup product, and fix $f$ as in the proof of Lemma~\ref{iden}. Then the coupled Morse differential $\overline{\partial}_{*}: \overline{C}_{*}(\mathbb{T}, D_{-, \mathfrak{q}}, f; \mathbb{Z}) \to \overline{C}_{*}(\mathbb{T}, D_{-, \mathfrak{q}}, f; \mathbb{Z})$ (with $\mathbb{Z}$ coefficients) is the zero map. Thus, we may identify the chain complex with its homology groups. \qed
\end{lemma}

We will assume henceforth that $f$ is chosen as in the proof of Lemma~\ref{iden}. Now, since $(\overline{\iota}^{\mathbb{Z}})^2 \sim \text{Id}$, it induces the identity on homology, which by Lemma~\ref{difftriv} implies that $(\overline{\iota}^{\mathbb{Z}})^2 = \text{Id}$ on the chain level. For the calculation which follows, we omit the $\mathbb{Z}$ label; assume we are using $\mathbb{Z}$ coefficients unless otherwise indicated. We have:
\[\text{Id} = (\overline{\iota})^2 = (\text{Id} + U^{-1}\overline{\iota}_{2} + U^{-2}\overline{\iota}_{4} + \dots)^{2}\]
\[= \text{Id} + U^{-1}(2\overline{\iota}_{2}) + U^{-2}(\overline{\iota}_{2}^{2} + 2\overline{\iota}_{4}) + \dots\]
\[ = \text{Id} + \sum_{i > 0}\sum_{\substack{(j_1, \dots, j_m)\\{j_1 + \dots + j_m = i}}} \left(\prod_{k = 1}^{m}\overline{\iota}_{2j_k}\right)U^{-i}\]
Since the coefficient of $U^{-i}$ in this expression is a homogeneous map of degree $-2i$, we can equate coefficients of $U^{-i}$ on either side of the equation, i.e. we must have that all of the coefficients vanish. In other words:
\begin{equation}\label{important} \sum_{\substack{(j_1, \dots, j_m)\\{j_1 + \dots + j_m = i}}} \left(\prod_{k = 1}^{m}\overline{\iota}_{2j_k}\right) = 0 \end{equation}
for all $i > 0$. We proceed by (strong) induction to show that $\overline{\iota}_{2j} = 0$ for all $j > 0$. Equation~\ref{important} in the case of $i = 1$ reads:
\[2\overline{\iota}_{2} = 0.\]
Since we are working on the chain level, and the chain complex $\overline{C}_{*}(\mathbb{T}, D_{-, \mathfrak{q}}, f; \mathbb{Z})$ has no 2-torsion, it follows that $\overline{\iota}_{2} = 0$.

Now, assuming $\iota_{2j} = 0$ for all $0 < j < i$, equation~\ref{important} reads:
\[\sum_{\substack{(j_1, \dots, j_m)\\{j_1 + \dots + j_m = i}}} \left(\prod_{k = 1}^{m}\overline{\iota}_{2j_k}\right) = 2\overline{\iota}_{2i} = 0\]
since every term in the sum other than $2\overline{\iota}_{2i}$ contains $\overline{\iota}_{2j}$ for some $0 < j < i$. We conclude that $\overline{\iota}^{\mathbb{Z}} = \text{Id}$, which completes the proof of the statement. \end{proof}

We remark that that above proof goes through in the more general case that $\overline{\HM}(Y, \mathfrak{s}; \mathbb{Z})$ has no 2-torsion. We may now prove Theorem~\ref{SWstabadjintro}.

\begingroup
\def\thetheorem{\ref{SWstabadjintro}}
\begin{theorem} Let $X$ be a smooth, closed four-manifold with $b_{2}^{+}(X) > 2$. Suppose that there exists a disjoint pair of smoothly embedded surfaces $S_1, S_2 \subset X\#(S^2 \times S^2)$ such that
\[[S_{1}]^2 > \max\{2g(S_{1}) - 2, 0\}\]
and
\[[S_{2}]^2 > \max\{2g(S_{2}) - 2, 0\}.\]
Then $SW_{X, \mathfrak{s}} = 0$ mod 2 for all spin structures $\mathfrak{s}$ on $X$.
\end{theorem}
\addtocounter{theorem}{-1}
\endgroup
\begin{proof} Let $Y_{1}$ and $Y_{2}$ be the boundaries of tubular neighborhoods of $S_{1}$ and $S_{2}$, respectively, and let $N$ denote their connected sum (obtained by taking the boundary of a thickened path from $Y_{1}$ to $Y_{2}$).  Note that $N$ is an involutively admissible cut of $X \#(S^2 \times S^2)$, so it suffices to show that $\HMI_{\text{red}}(N) = 0$. By work of Ozsváth-Szabó on the reduced Floer homology of circle bundles over surfaces (see \cite{MR2377279}, Theorem 5.6), we know that $\HF_{\text{red}}(Y_{i}) = 0$. By the connected sum formula for Heegaard Floer homology, it follows that $\HF_{\text{red}}(N) = 0$.  Furthermore, by work of Kutluhan-Lee-Taubes establishing isomorphisms between Heegaard Floer homology and monopole Floer homology \cite{MR4194305}, we have $\HM_{\text{red}}(N) \cong \HF_{\text{red}}(N) = 0$.

Now, we claim that the triple cup product of $N$ vanishes. Indeed, this follows from the fact that the triple cup product of $Y_{1}$ and $Y_{2}$ both vanish, seeing as the map on first cohomology $\pi_{i}^{*}: H^1(S_i; \mathbb{Z}) \to H^1(Y_{i}; \mathbb{Z})$ induced by the $S^1$-bundle projection map $\pi_{i}: Y_{i} \to S_{i}$ is surjective, and the triple cup product on $H^1(S_{i}; \mathbb{Z})$ is trivial since $S_{i}$ is a surface. Thus, Proposition~\ref{tech} implies that $\overline{\iota}_{*} = \text{Id}$. Combined with the fact that $\HM_{\text{red}}(N) = 0$, this implies that $\HMI_{\text{red}}(N) = 0$ by Proposition~\ref{vanishred}, completing the proof of the statement.\end{proof}

\bibliographystyle{plain}
\bibliography{references2}

\bigskip

\noindent{\small \textsc{Department of Mathematics, Harvard University \\ Cambridge, Massachusetts, 02138}\\
\textit{Email}: \texttt{obrass@math.harvard.edu}}
\end{document}